\documentclass{amsart}

\usepackage[margin=1.1in]{geometry}
\usepackage{amsmath,amssymb,amsfonts,amsthm,mathtools}
\usepackage{xcolor}
\usepackage{hyperref}

\definecolor{darkblue}{rgb}{0.0,0.0,0.45}
\hypersetup{
 colorlinks=true,
 linkcolor=darkblue,
 citecolor=darkblue,
 urlcolor=darkblue,
 pdftitle={Existence of classical solutions to the exterior Dirichlet problem for Hessian quotient equations},
 pdfauthor={Yuxuan Liao and Jiguang Bao},
 pdfsubject={Exterior Dirichlet problems for Hessian quotient equations},
 pdfkeywords={Hessian quotient equation, exterior Dirichlet problem, classical admissible solution, prescribed asymptotic Hessian, integral tail condition}
}

\numberwithin{equation}{section}

\theoremstyle{plain}
\newtheorem{theorem}{Theorem}[section]
\newtheorem{proposition}[theorem]{Proposition}
\newtheorem{lemma}[theorem]{Lemma}
\newtheorem{corollary}[theorem]{Corollary}
\newtheorem*{claim*}{Claim}

\theoremstyle{definition}
\newtheorem{definition}[theorem]{Definition}

\theoremstyle{remark}
\newtheorem{remark}[theorem]{Remark}

\title[Exterior Hessian quotient equations]{Existence of classical solutions to the exterior Dirichlet problem for Hessian quotient equations}

\author{Yuxuan Liao}
\address{School of Mathematical Sciences, Beijing Normal University, Beijing 100875, China}
\email{yxliao@mail.bnu.edu.cn}

\author{Jiguang Bao}
\address{School of Mathematical Sciences, Beijing Normal University, Beijing 100875, China}
\email{jgbao@bnu.edu.cn}
\thanks{Jiguang Bao was supported by the National Natural Science Foundation of China (Grant No.~12371200) and the Beijing Natural Science Foundation (Grant No.~1254049).}

\date{}

\begin{document}

\begin{abstract}
        This paper studies the exterior Dirichlet problem for Hessian quotient equations with nonconstant right-hand sides. We prove the existence of classical admissible solutions with prescribed asymptotic Hessians and establish convergence of the Hessian at infinity. The main difficulties are obtaining second-order estimates on expanding annuli that are uniform in the outer radius and deriving Hessian convergence under an integral tail condition with no prescribed decay rate. These are resolved through a radius-independent boundary-to-interior estimate and a blow-down argument. We also allow nonradial perturbations of the source. Under stronger pointwise assumptions, we obtain higher-order asymptotic expansions and solutions for every sufficiently large prescribed asymptotic constant.
\end{abstract}

\maketitle

\section{Introduction and main results}\label{sec:introduction}

We study the exterior Dirichlet problem for the Hessian quotient equation
\begin{equation}\label{eq:exterior-problem}
        \begin{cases}
                S_{k,l}(D^2u):=\dfrac{S_k(D^2u)}{S_l(D^2u)}=g(x)
                       & \text{in }E:=\mathbb R^n\setminus\overline\Omega, \\[4pt]
                u=\phi & \text{on }\partial\Omega,                         \\
                \lambda(D^2u)\in\Gamma_k & \text{in }E,
        \end{cases}
\end{equation}
where $n\ge3$, $0\le l<k\le n$, $D^2u$ denotes the Hessian of $u$, and, for $A\in\operatorname{Sym}(n)$,
\begin{align*}
        S_m(A):=\sigma_m(\lambda(A)),\qquad m=0,1,\ldots,n, \\
        \Gamma_m:=\{\lambda\in\mathbb R^n:\sigma_j(\lambda)>0,\ 1\le j\le m\}.
\end{align*}
Here $\sigma_m$ is the $m$-th elementary symmetric polynomial of the eigenvalues
$\lambda(A)$, with $\sigma_0=1$, and $\Gamma_m$ is the G\r{a}rding cone, an open
convex symmetric cone with vertex at the origin.

The study of fully nonlinear elliptic equations such as \eqref{eq:exterior-problem} originates from the search for specific K\"ahler metrics on complex manifolds, where the ratio of the Ricci curvature to the K\"ahler form is prescribed \cite{CNS85, Trudinger95}. As a natural generalization, the Hessian quotient equation \eqref{eq:exterior-problem} reduces to the Monge--Amp\`ere equation when $(k,l)=(n,0)$, to the $k$-Hessian equation when $l=0$, and to the Poisson equation when $(k,l)=(1,0)$.

The exterior Dirichlet problem was first introduced by Meyers and Serrin \cite{MeyersSerrin1960} for linear elliptic equations. In the fully nonlinear setting, Caffarelli and Li \cite{CaffarelliLi2003} established the solvability of the exterior Monge--Amp\`ere equation. In view of the J\"orgens--Calabi--Pogorelov theorem \cite{Jorgens1954,Calabi1958,Pogorelov1972}, they proved that for $n\ge3$, prescribed $A\in\operatorname{Sym}(n)$ with $\det A=1$ and $b\in\mathbb R^n$, and all sufficiently large constants $c$, the exterior problem admits a unique classical solution with
\begin{equation*}
        u(x)=\frac12x^TAx+b\cdot x+c+O(|x|^{2-n}).
\end{equation*}
Building on this foundation, Li and Lu \cite{LiLu2018} completed the existence/nonexistence characterization in terms of the prescribed asymptotic behavior. Further developments relaxed the boundary data requirements; Bao and Wang \cite{BaoWang2024}, for example, characterized solvability in terms of the boundary value being semiconvex with respect to the inner boundary. Related techniques were also adapted to construct entire solutions with prescribed asymptotics \cite{BaoXiongZhou2019}, while various extensions addressing two-dimensional domains, nonhomogeneous terms, and refined error estimates can be found in \cite{FerrerMartinezMilan1999,Delanoe1992,BaoLiZhang2015,BaoLiZhang2016,LiuBao2023} and the references therein.

For $k$-Hessian equations, viscosity solutions with prescribed asymptotics were first constructed by Dai and Bao \cite{DaiBao2011}. Bao, Li, and Li \cite{BaoLiLi2014} then treated general symmetric asymptotics, obtaining
\begin{equation*}
        u(x)=\frac12x^TAx+b\cdot x+c+O(|x|^{\theta(2-n)}),
\end{equation*}
where $\theta(n,k,A)\in[\frac{k-2}{n-2},1]$. This framework also led to entire solutions \cite{WangBao2022,LiDai2020}. More recently, Li and Xiao \cite{LiXiao2026} obtained classical smooth strictly $k$-convex exterior solutions through estimates independent of the truncation radius. Building on this annular scheme and the metric selected by the linearized operator, Bao and Jiang \cite{BaoJiang2026} obtained viscosity and smooth solutions for arbitrary $A$ satisfying $S_k(A)=1$ for which the operator is elliptic.

For Hessian quotient equations, constant-source solvability with symmetric and generalized symmetric asymptotics was developed by Dai \cite{Dai2011JMAA}, Li and Dai \cite{LiDai2012}, and Li and Li \cite{LiLi2018}; see also \cite{LiLiZhao2025}. For nonconstant sources, Jiang, Li, and Li \cite{JiangLiLi2022ExteriorQuotient} constructed viscosity solutions under the pointwise assumption $g(x)=1+O(|x|^{-\beta})$, $\beta>2$, and Dai, Bao, and Wang \cite{DaiBaoWang2025} established viscosity solvability when $g$ is a perturbation of a generalized symmetric function at infinity.

Against this background, we prove classical solvability for Hessian quotient equations for every prescribed $A\in\operatorname{Sym}(n)$ with $\lambda(A)\in\Gamma_k$ and $S_{k,l}(A)=1$ under an $A$-adapted integral tail condition, together with $Du-Ax\to b$ and $D^2u\to A$ without a prescribed power rate. Under stronger pointwise assumptions, we obtain higher-order asymptotics and solutions for every sufficiently large prescribed asymptotic constant.

Before stating the main results, we record the remaining definitions used below.
\begin{definition}
        A function $u\in C^2(\Omega)$ is called $k$-admissible (or strictly
        $k$-convex) if $\lambda(D^2u(x))\in\Gamma_k$ for every $x\in\Omega$.
        It is weakly $k$-admissible (or $k$-convex) if
        $\lambda(D^2u(x))\in\overline{\Gamma}_k$ for every $x\in\Omega$.
\end{definition}
Let $\nu$ denote the unit normal to $\partial\Omega$ pointing from
$\Omega$ into $E$.
\begin{definition}
        The domain $\Omega$ is strictly star-shaped with respect
        to the origin if $x\cdot\nu(x)>0$ on $\partial\Omega$, and it is strictly
        $(k-1)$-convex ($(k-1)$-convex) if the principal-curvature vector of $\partial\Omega$, computed with respect to $\nu$, belongs to $\Gamma_{k-1}$ ($\overline{\Gamma}_{k-1}$); when $k=1$ the latter condition is void.
\end{definition}

We next formulate the tail condition used in the main theorem. Fix $A\in\operatorname{Sym}(n)$ with $\lambda(A)\in\Gamma_k$ and $S_{k,l}(A)=1$. Ellipticity of $S_{k,l}$ at $A$ gives a
unique positive definite matrix $G_A$ such that
\begin{equation}\label{eq:intro-linearized-matrix}
        D S_{k,l}(A)[H]=(k-l)\operatorname{tr}(G_AH),
        \qquad H\in\operatorname{Sym}(n).
\end{equation}
The linearization selects the $A$-adapted radius
\begin{equation}\label{eq:intro-affine-coordinate}
        |x|_A:=\bigl(x^TG_A^{-1}x\bigr)^{1/2}.
\end{equation}
This radius is used for the radial part of the source. Since $G_A>0$,
$|x|_A\asymp|x|$, so estimates that involve only decay orders may be stated using
the Euclidean radius $|x|$.

\emph{The $A$-adapted integral tail condition.}
For a source $g(x)\to1$ as $|x|\to\infty$, we consider decompositions into an
$A$-adapted radial profile and a remainder: let $R_0>0$,
$g_0\in C^2([R_0,\infty))$ be positive, and
$\varepsilon\in C^2([R_0,\infty))$ be nonnegative such that
\begin{equation}\label{eq:g-tail-decomposition}
        \bigl|g(x)-g_0(|x|_A)\bigr|
        \le \varepsilon(|x|_A),
        \qquad |x|_A\ge R_0,
\end{equation}
with
\[
        g_0(r)\longrightarrow1,
        \qquad \varepsilon(r)\longrightarrow0
        \quad\text{as }r\to\infty.
\]
Such a decomposition is not unique. When $\varepsilon\equiv0$, we call the source
\emph{$A$-adapted radial}. Given a decomposition as above, set
\begin{equation}\label{eq:g-tail-envelope}
        m_g(r):=|g_0(r)-1|+\varepsilon(r),
        \qquad
        \widehat m_g(r):=m_g(r)+r^{-n}\int_{R_0}^r s^{n-1}m_g(s)\,ds.
\end{equation}
We say that $g$ satisfies the $A$-adapted integral tail condition $(\mathrm H_A)$ if
there exists such a decomposition for which
\begin{equation}\label{eq:g-tail-integrability}
        \int_{R_0}^{\infty}r\widehat m_g(r)^2\,dr<\infty,
        \qquad
        \int_{R_0}^{\infty}r\varepsilon(r)\,dr<\infty.
\end{equation}
\begin{remark}
        The two conditions in \eqref{eq:g-tail-integrability} control the radial tail
        and the angular oscillation, respectively. The first condition allows radial tails of size $O((r\log r)^{-1})$, which is weaker than $O(r^{-1-\delta})$ for any $\delta>0$.
        The role of the second condition is not visible from a bound on $|g-1|$ alone.
        Indeed, set
        \[
                a:=\left({\binom nl}/{\binom nk}\right)^{\frac1{k-l}},
                \qquad A:=aI,
        \]
        and, for large $r=|x|$, consider
        \[
                u_1(x):=\frac a2|x|^2
                +\frac{na}{(k-l)(n-1)}\frac r{\log r},
                \qquad
                u_2(x):=\frac a2|x|^2
                +\frac{a}{k-l}x_1\log\log r.
        \]
        Both satisfy $D^2u_i\to A$ and
        $|S_{k,l}(D^2u_i)-1|=O((r\log r)^{-1})$. The source associated with $u_1$
        is $A$-adapted radial, satisfies $(\mathrm H_A)$, and has $Du_1-Ax\to0$;
        the source associated with $u_2$ has the same pointwise size but fails the
        second condition, and $Du_2-Ax$ does not converge. Thus the size of $g-1$
        alone does not capture the angular stability relevant to the prescribed linear
        asymptotics; see Appendix~\ref{app:angular-instability} for details.
\end{remark}

\begin{theorem}\label{thm:main}
        Let $n\ge3$, $0\le l<k\le n$, and $q\ge4$ be integers, and let $0<\alpha<1$.
        Let $\Omega\subset\mathbb R^n$ be a bounded $C^{q,\alpha}$ domain which is
        strictly star-shaped with respect to the origin and strictly
        $(k-1)$-convex. Let
        \[
                \phi\in C^{q,\alpha}(\partial\Omega),
                \qquad g\in C^{q-2,\alpha}(\mathbb R^n),
        \]
        and assume that $\inf_{\mathbb R^n}g>0$ and $\|g\|_{C^2(\mathbb R^n)}<\infty$.

        Fix any $A\in\operatorname{Sym}(n)$ with $\lambda(A)\in\Gamma_k$ and $S_{k,l}(A)=1$, and assume that $g$ satisfies
        $(\mathrm H_A)$. Then, for every $b\in\mathbb R^n$, problem
        \eqref{eq:exterior-problem} admits a classical $k$-admissible solution $u$.
        Moreover,
        \[
                u\in C^{q,\alpha}_{\mathrm{loc}}(\overline E),
        \]
        and
        \begin{equation}\label{eq:exterior-problem-linear-asymptotic}
                u(x)-\frac12x^TAx-b\cdot x=o(|x|),
                \qquad
                Du(x)-Ax\longrightarrow b,
                \qquad
                D^2u(x)\longrightarrow A
                \quad\text{as }|x|\to\infty.
        \end{equation}
\end{theorem}

A source of size $|x|^{-s}$ produces an exterior Newton potential with one of three familiar scales: the local scale $|x|^{2-s}$ for $2<s<n$, the critical scale $|x|^{2-n}\log|x|$ at $s=n$, or the capacity scale $|x|^{2-n}$ for $s>n$. We record these three possibilities as
\begin{equation}\label{eq:potential-scale}
        \mathcal P_s(r):=
        \begin{cases}
                r^{2-s},       & 2<s<n, \\
                r^{2-n}\log r, & s=n,   \\
                r^{2-n},       & s>n.
        \end{cases}
\end{equation}

\begin{theorem}\label{thm:asymptotic}
        Assume the hypotheses of Theorem~\ref{thm:main}. Suppose moreover that $g_0$
        in $(\mathrm H_A)$ can be chosen so that, for some $R_1\ge R_0$,
        \begin{equation}\label{eq:g-weighted-tail-bounds}
                |(g_0-1)^{(j)}(r)|\le Cr^{-\gamma-j},
                \qquad
                \bigl|D^j\bigl(g(x)-g_0(|x|_A)\bigr)\bigr|
                \le C|x|^{-\beta-j},
                \qquad 0\le j\le N,
        \end{equation}
        for $r\ge R_1$ and $|x|_A\ge R_1$, where
        $1\le N\le q-2$, $\gamma>1$, and $\beta>2$.
        In the $A$-adapted radial case, the remainder $g(x)-g_0(|x|_A)=0$ and we use the convention $\beta=\infty$.
        Let $u$ be an exterior solution produced by the construction in the proof of
        Theorem~\ref{thm:main} with this choice of $g_0$. Then, for some constant
        $c_\infty\in\mathbb R$, the remainder
        \begin{equation}\label{eq:refined-decomposition}
                \tau(x):=
                u(x)-\frac12x^TAx-b\cdot x-c_\infty
        \end{equation}
        satisfies the higher-order decay estimates
        \begin{equation}\label{eq:refined-derivative-decay}
                \limsup_{|x|\to\infty}
                \frac{|x|^j|D^j(\tau-\Phi_A)(x)|}
                {\mathcal P_{s_*}(|x|)}<\infty,
                \qquad 0\le j\le N+1,
                \qquad s_*:=\min\{\beta,2\gamma\},
        \end{equation}
        where $\Phi_A(x)=\varphi(|x|_A)$ is the radial correction given by
        \begin{equation}\label{eq:intro-linear-particular}
                \varphi(r)
                :=\frac{1}{(n-2)(k-l)}
                \int_{R_0}^{r}s\left[1-\left(\frac{s}{r}\right)^{n-2}\right]
                \bigl(g_0(s)-1\bigr)\,ds.
        \end{equation}
        It satisfies
        \begin{equation}\label{eq:intro-linear-response}
                \operatorname{tr}(G_AD^2\Phi_A)
                =\frac{g_0(|x|_A)-1}{k-l}.
        \end{equation}
\end{theorem}

\begin{corollary}\label{cor:polynomial-asymptotics}
        Under the hypotheses of Theorem~\ref{thm:asymptotic}, assume further that
        $\gamma>2$. Since $\Phi_A$ then has a finite limit, we normalize it by an
        additive constant so that
        \begin{equation}\label{eq:phiA-vanishing-normalization}
                \Phi_A(x)\longrightarrow0
                \qquad\text{as }|x|\to\infty.
        \end{equation}
        There exists $c_*\in\mathbb R$ such that, for every $c>c_*$,
        problem~\eqref{eq:exterior-problem} admits a unique classical
        $k$-admissible solution $u_c$ with
        \[
                \tau_c(x):=u_c(x)-\frac12x^TAx-b\cdot x-c
        \]
        satisfying, for $\mu:=\min\{\gamma,\beta\}>2$,
        \begin{equation}\label{eq:power-tail-consequences}
                \limsup_{|x|\to\infty}
                \frac{|x|^j|D^j\tau_c(x)|}{\mathcal P_\mu(|x|)}<\infty,
                \qquad 0\le j\le N+1.
        \end{equation}
\end{corollary}
\begin{remark}
	Notice that $s_*$ and $\mu$ concern different remainders: $s_*$ controls $\tau-\Phi_A$, after the radial correction has been removed, whereas $\mu$ controls $\tau_c$ itself. Accordingly, $\mu=\min\{\gamma,s_*\}=\min\{\gamma,\beta\}$.
\end{remark}

Compared with existing results for nonconstant-source Hessian quotient equations, the present theorems extend the solvability and asymptotic theory in several directions.
Indeed, taking $g_0\equiv1$ includes the pointwise class
$g(x)=1+O(|x|^{-\beta})$, $\beta>2$, of
\cite{JiangLiLi2022ExteriorQuotient}; in the common quadratic-asymptotic regime we also
remove the additional restriction on $A$ in \cite{DaiBaoWang2025} and obtain classical
solutions under the present smooth assumptions. The special cases $l=0$, $g\equiv1$
and $(k,l)=(n,0)$, $g_0\equiv1$ recover, respectively, the quadratic $k$-Hessian regime
of \cite{BaoLiLi2014}, with the capacity-scale remainder, and the corresponding
fast-decay Newton-potential scales in the Monge--Amp\`ere theory
\cite{BaoLiZhang2015,LiuBao2023}.

The proof has two main ingredients. First, boundary estimates together with a radius-independent boundary-to-interior estimate yield uniform $C^2$ bounds on the expanding annuli. Second, the integral tail condition gives only qualitative control at infinity, and a blow-down argument yields $D^2u\to A$.

The paper is organized as follows.
Section~\ref{sec:construction} introduces the concave operator $F$ and constructs the global comparison functions.
Section~\ref{sec:estimates} derives estimates on
expanding annuli that are independent of the outer radius. Section~\ref{sec:compactness} constructs the exterior solution, recovers its regularity at the fixed boundary, and then proves the quantitative expansion under the additional pointwise tail assumptions.
Appendix~\ref{app:angular-instability} illustrates the role of angular stability in the prescribed linear asymptotics. Appendix~\ref{app:regular-variation} records the decay of radial potentials for regularly varying tails, while Appendix~\ref{app:annular-solvability} records the bounded-annulus solvability result used in the exhaustion.

\section{The concave operator and construction of global sub- and supersolutions}
\label{sec:construction}
In this section we introduce the concave operator $F$ and the geometric facts needed near the obstacle. The comparison functions are then constructed in three stages: first a pair of far-field barriers, next a fixed-scale connection to the obstacle, and finally the assembly on expanding annuli.

Throughout the construction, we use the following notation. Unsubscripted profiles refer to the corresponding untruncated functions on the tail; a subscript $R$ denotes truncation at radius $R$, and a subscript $j$ refers to the exhaustion sequence $R_j\to\infty$. Superscripts $+$ and $-$ distinguish upper and lower comparison functions. The parameter $c$ appears explicitly in $P_c$; after it is fixed, its dependence in the comparison functions is usually suppressed. Exhaustion-dependent scalar normalizations are denoted by $c_j$, while $c_\infty$ is reserved for the final asymptotic constant.

\subsection{The concave operator and fixed-boundary geometry}
\label{subsec:fixed-boundary-geometry}

For the construction it is convenient to replace the quotient by its homogeneous concave root,
\[
        F(M):=\left(\frac{S_k(M)}{S_l(M)}\right)^{1/(k-l)},
        \qquad f:=g^{1/(k-l)}.
\]
Then \eqref{eq:exterior-problem} is equivalent to
\[
        F(D^2u)=f\quad\hbox{in }E,
        \qquad u=\phi\quad\hbox{on }\partial\Omega.
\]
The operator $F$ is homogeneous of degree one, elliptic, and concave on the set of symmetric matrices whose eigenvalues lie in $\Gamma_k$ \cite{CNS85}. Fix the matrix $A$ and the vector $b$ from
Theorem~\ref{thm:main}. By \eqref{eq:intro-linearized-matrix},
\[
        DF(A)[H]=\operatorname{tr}(G_AH).
\]
Throughout this paper, write
\[
        G:=G_A,
        \qquad \mathcal Lv:=\operatorname{tr}(GD^2v),
        \qquad \rho(x):=(x^TG^{-1}x)^{1/2},
\]
and set
\[
        P_c(x):=\frac12x^TAx+b\cdot x+c.
\]
The parameter $c$ will be chosen sufficiently large in the annular construction.

Strict star-shapedness allows us to write
\[
        \partial\Omega
        =\{\varrho_\Omega(\theta)\theta:\theta\in\mathbb S^{n-1}\},
        \qquad \varrho_\Omega>0.
\]
The associated Minkowski functional is
\[
        \beta_\Omega(r\theta):=\frac{r}{\varrho_\Omega(\theta)}.
\]
It is smooth on every fixed annular neighborhood of $\partial\Omega$,
$\beta_\Omega=1$ on $\partial\Omega$, and its level sets are the homothetic
hypersurfaces $s\partial\Omega$. In particular, $\nabla\beta_\Omega$ does not
vanish there. If $\xi$ and $\zeta$ are tangent to a level set of
$\beta_\Omega$, then
\[
        D^2\beta_\Omega(\xi,\zeta)
        =|\nabla\beta_\Omega|\,\mathrm{II}(\xi,\zeta).
\]
Thus strict $(k-1)$-convexity of $\partial\Omega$ gives uniform positivity of
the first $k-1$ elementary symmetric functions of the tangential Hessian on
every fixed homothetic ring.

Let $d_\Omega$ denote the exterior distance to $\partial\Omega$. On a fixed
tubular collar, $d_\Omega$ is smooth and its level sets are the outer parallel
hypersurfaces of $\partial\Omega$. After shrinking the collar, their
principal-curvature vectors remain in a fixed compact subset of
$\Gamma_{k-1}$. These properties will be used in the fixed-boundary construction below.

\subsection{The far-field sub- and supersolution pair}

For an $A$-adapted radial function, direct differentiation gives
\[
        \mathcal L(h\circ\rho)
        =h''(\rho)+\frac{n-1}{\rho}h'(\rho).
\]
An $A$-adapted radial source is inverted by first measuring its enclosed radial
mass and then integrating once more in the radial direction. We encode these
two operations by
\[
        \mathcal A h(r):=r^{-n}\int_{R_0}^r s^{n-1}h(s)\,ds,
        \qquad
        \mathcal Hh(r):=\int_{R_0}^r t\,\mathcal Ah(t)\,dt.
\]
By construction, $\mathcal Hh$ solves the radial linearized equation
\[
        (\mathcal Hh)''+\frac{n-1}{r}(\mathcal Hh)'=h
\]
with zero initial data at $R_0$; this explicit integral representation and its
normalized derivative $\mathcal Ah(r)=(\mathcal Hh)'(r)/r$ are used in
Lemma~\ref{lem:far-field-barriers} to cancel the linear error of the
perturbed polynomial $P_c$ and bound the remaining nonlinear terms.

We first consider an $A$-adapted radial source. Let $f_0(r)\to1$, and choose a
nonnegative envelope $m$ such that $|f_0-1|\le m$ and $m(r)\to0$. Define $\widehat m:=m+\mathcal A m$, and assume
\begin{equation}\label{eq:radial-square-integrability}
        \int_{R_0}^{\infty}r\widehat m(r)^2\,dr<\infty.
\end{equation}
Fix $0<\theta_0<\theta_1<1$ and a smooth cutoff $\chi$ which equals one on
$[0,\theta_0]$ and zero on $[\theta_1,\infty)$. For $R>R_0$, let
\[
        \chi_R(r):=\chi(r/R),
        \qquad f_{0,R}(r):=1+\chi_R(r)(f_0(r)-1),
\]
\[
        \Phi(x):=[\mathcal H(f_0-1)](\rho(x)),
        \qquad
        \Phi_R(x):=[\mathcal H(f_{0,R}-1)](\rho(x)),
\]
and
\[
        \Psi_R(r):=-\int_r^R t^{1-n}
        \int_{R_0}^t s^{n-1}\widehat m(s)^2\,ds\,dt,
        \qquad
        \Psi(r):=\lim_{R\to\infty}\Psi_R(r).
\]

\begin{lemma}\label{lem:far-field-barriers}
        After increasing $R_0$, there is a constant $C_H>0$, independent of $R$, such
        that
        \[
                \overline u_R^\infty:=P_c+\Phi_R,
                \qquad
                \underline u_R^\infty:=P_c+\Phi_R+C_H\Psi_R(\rho)
        \]
        are $k$-admissible on $\{R_0\le\rho\le R\}$ and satisfy
        \[
                F(D^2\underline u_R^\infty)\ge f_{0,R}(\rho),
                \qquad
                F(D^2\overline u_R^\infty)\le f_{0,R}(\rho).
        \]
        The two branches are ordered and agree on $\{\rho=R\}$, with
        \[
                0\ge \underline u_R^\infty-\overline u_R^\infty
                =C_H\Psi_R(\rho)\ge C_H\Psi(\rho).
        \]
        Moreover, $\Phi_R=\Phi$ on $\{\rho\le\theta_0R\}$, while
        $\Psi(r)\to0$ and $\Psi(r)/r^2\to0$ as $r\to\infty$.
\end{lemma}

\begin{proof}
        The function $\Phi_R$ cancels the radial source in the linearization at $A$.
        Concavity gives the supersolution inequality, while $\Psi_R$ provides the
        positive linearized term needed for the subsolution inequality.

        \smallskip
        \noindent\emph{Preliminaries.}
        Define $H_m:=\mathcal Hm$.
        Tonelli's theorem gives the exact identity
        \[
                -\Psi(R_0)
                =\frac1{n-2}\int_{R_0}^{\infty}r\widehat m(r)^2\,dr.
        \]
        Hence $\Psi(R_0)$ is finite under \eqref{eq:radial-square-integrability}, and
        $\Psi(r)\to0$ by monotone convergence. The same condition also gives
        \[
                H_m(r)=o(r).
        \]
        Indeed, for fixed $R>R_0$ and $r>R$, the inequality
        $\mathcal A m\le\widehat m$ and Cauchy--Schwarz give
        \[
                \frac{H_m(r)}r
                \le \frac{H_m(R)}r
                +\frac1r
                \left(\int_R^r t\widehat m(t)^2\,dt\right)^{1/2}
                \left(\int_R^r t\,dt\right)^{1/2}.
        \]
        First let $r\to\infty$ and then $R\to\infty$.

        \smallskip
        \noindent\emph{The supersolution}
        Since $|f_{0,R}-1|\le m$, the radial identities imply
        \[
                \frac{|(\mathcal H(f_{0,R}-1))'(r)|}{r}
                \le\mathcal A m(r),
        \]
        and
        \[
                |(\mathcal H(f_{0,R}-1))''(r)|
                \le m(r)+(n-1)\mathcal A m(r).
        \]
        Differentiating $\rho$ therefore gives, uniformly in $R$,
        \[
                |D^2\Phi_R(x)|\le C\widehat m(\rho(x)),
                \qquad
                \mathcal L\Phi_R=f_{0,R}(\rho)-1,
                \qquad
                |\Phi_R(x)|\le H_m(\rho(x)).
        \]
        After increasing $R_0$, the eigenvalues of all matrices $A+D^2\Phi_R$
        remain in one fixed compact subset of $\Gamma_k$.

        The tangent plane to the concave operator $F$ at $A$ lies above its graph.
        Hence
        \[
                F(D^2\overline u_R^\infty)
                =F(A+D^2\Phi_R)
                \le1+DF(A)[D^2\Phi_R]
                =f_{0,R}(\rho).
        \]

        \smallskip
        \noindent\emph{The subsolution}
        On the corresponding compact matrix set,
        \[
                F(A+H)\ge1+DF(A)[H]-C_F|H|^2.
        \]
        Let $\delta:=\sup_{r\ge R_0}m(r)$. The required Hessian estimate is a
        direct calculation. If
        \[
                M(r):=\int_{R_0}^r s^{n-1}m(s)\,ds,
        \]
        then
        \[
                \int_{R_0}^r s^{n-1}m(s)^2\,ds\le\delta M(r)
        \]
        and, differentiating $M(s)^2s^{-n}$ and using $M(R_0)=0$,
        \[
                \begin{aligned}
                        n\int_{R_0}^r M(s)^2s^{-n-1}\,ds
                         & =2\int_{R_0}^r M(s)M'(s)s^{-n}\,ds-M(r)^2r^{-n} \\
                         & \le2\int_{R_0}^r M(s)M'(s)s^{-n}\,ds.
                \end{aligned}
        \]
        The discarded upper-boundary contribution is nonpositive in the displayed
        identity. Since $M(s)s^{-n}=\mathcal A m(s)\le\delta/n$, we obtain
        \[
                \int_{R_0}^r s^{n-1}(\mathcal A m(s))^2\,ds
                =\int_{R_0}^r M(s)^2s^{-n-1}\,ds
                \le C\delta M(r).
        \]
        Consequently,
        \[
                \mathcal A(\widehat m^2)(r)\le C\delta\widehat m(r).
        \]
        The radial equation for $\Psi_R$, together with
        $\widehat m\le C\delta$, now gives
        \[
                |D^2\Psi_R(\rho)|\le C\delta\widehat m(\rho).
        \]
        Choose $C_H$ first and then move $R_0$ outward so that
        \[
                C_F|D^2\Phi_R+C_HD^2\Psi_R(\rho)|^2
                \le C_H\widehat m(\rho)^2
        \]
        and their eigenvalues remain in the chosen compact subset of $\Gamma_k$. Since
        \[
                \mathcal L\bigl(C_H\Psi_R(\rho)\bigr)
                =C_H\widehat m(\rho)^2,
        \]
        Taylor's formula yields
        \[
                F(D^2\underline u_R^\infty)
                \ge f_{0,R}(\rho)+C_H\widehat m(\rho)^2
                -C_H\widehat m(\rho)^2
                =f_{0,R}(\rho).
        \]

        Finally, $\Psi_R\le0$, $\Psi_R(R)=0$, and $\Psi\le\Psi_R$.
        The finiteness and decay of $\Psi$ were established above from the Tonelli
        identity. The cutoff is inactive for $r\le\theta_0R$, so $\Phi_R=\Phi$
        there.
\end{proof}

Now return to the general condition $(\mathrm H_A)$. Let $g_0$ and
$\varepsilon$ be as in \eqref{eq:g-tail-decomposition}--
\eqref{eq:g-tail-integrability}, and put
\[
        f_0:=g_0^{1/(k-l)}.
\]
After increasing $R_0$, the root map is bi-Lipschitz on the relevant range.
Thus, after multiplying $\varepsilon$ by a fixed structural constant,
\[
        |f(x)-f_0(\rho(x))|\le\varepsilon(\rho(x)).
\]
For the rest of this subsection set
\[
        m:=|f_0-1|+\varepsilon,
        \qquad \widehat m:=m+\mathcal A m,
        \qquad H_m:=\mathcal Hm.
\]
The assumptions in $(\mathrm H_A)$ imply
\[
        \int_{R_0}^{\infty}r\widehat m(r)^2\,dr<\infty,
        \qquad
        \int_{R_0}^{\infty}r\varepsilon(r)\,dr<\infty.
\]
The functions $\Psi_R$ and $\Psi$ are now formed with this enlarged envelope.
For each truncation radius define
\[
        f_R(x):=1+\chi_R(\rho(x))(f(x)-1),
        \qquad
        f_{0,R}(r):=1+\chi_R(r)(f_0(r)-1),
\]
\[
        \Phi_R(x):=[\mathcal H(f_{0,R}-1)](\rho(x)),
\]
and
\[
        T(r):=\int_r^\infty t^{1-n}
        \int_{R_0}^t s^{n-1}\varepsilon(s)\,ds\,dt.
\]

\begin{lemma}\label{lem:general-far-field-barriers}
        After increasing $C_H$ and $R_0$ if necessary, the functions
        \[
                \overline u_R^\infty
                =P_c+\Phi_R+T(\rho),
        \]
        \[
                \underline u_R^\infty
                =P_c+\Phi_R-T(\rho)+2T(R)
                +C_H\Psi_R(\rho)
        \]
        are $k$-admissible on $\{R_0\le\rho\le R\}$ and satisfy
        \[
                F(D^2\underline u_R^\infty)\ge f_R,
                \qquad
                F(D^2\overline u_R^\infty)\le f_R.
        \]
        They are ordered, agree on $\{\rho=R\}$, and obey
        \[
                0\ge \underline u_R^\infty-\overline u_R^\infty
                \ge-2T(\rho)+C_H\Psi(\rho).
        \]
        In particular, $T(r)\to0$ and $\Psi(r)\to0$ as $r\to\infty$.
\end{lemma}

\begin{proof}
        The envelope gives
        \[
                f_{0,R}(\rho)-\varepsilon(\rho)
                \le f_R
                \le f_{0,R}(\rho)+\varepsilon(\rho).
        \]
        Moreover,
        \[
                \mathcal L\Phi_R=f_{0,R}(\rho)-1,
                \qquad
                \mathcal L(T\circ\rho)=-\varepsilon(\rho).
        \]
        Concavity therefore gives the upper inequality directly:
        \[
                F(D^2\overline u_R^\infty)
                \le f_{0,R}(\rho)-\varepsilon(\rho)
                \le f_R.
        \]
        For the lower branch, set
        \[
                H_R:=D^2\Phi_R-D^2(T\circ\rho)+C_HD^2\Psi_R.
        \]
        After increasing $R_0$, all matrices $A+H_R$ lie in one fixed compact
        neighborhood of $A$ whose eigenvalues lie in $\Gamma_k$. Taylor's formula on this
        neighborhood gives
        \[
                F(A+H_R)\ge 1+DF(A)[H_R]-C_F|H_R|^2.
        \]
        Since
        \[
                \mathcal L\Phi_R=f_{0,R}(\rho)-1,
                \qquad
                \mathcal L(-T\circ\rho)=\varepsilon(\rho),
                \qquad
                \mathcal L(C_H\Psi_R)=C_H\widehat m(\rho)^2,
        \]
        we obtain
        \begin{equation}\label{eq:general-lower-taylor}
                F(A+H_R)
                \ge f_{0,R}(\rho)+\varepsilon(\rho)
                +C_H\widehat m(\rho)^2-C_F|H_R|^2.
        \end{equation}
        Let
        \[
                \delta:=\sup_{r\ge R_0}m(r).
        \]
        The radial estimates from Lemma~\ref{lem:far-field-barriers} and the definition
        of $T$ give constants $C_0,C_1$, independent of $R$, such that
        \[
                |D^2\Phi_R|+|D^2(T\circ\rho)|\le C_0\widehat m(\rho),
                \qquad
                |D^2\Psi_R|\le C_1\delta\widehat m(\rho).
        \]
        Consequently
        \[
                |H_R|\le (C_0+C_1C_H\delta)\widehat m(\rho).
        \]
        Choose $C_H>2C_FC_0^2+1$ first. Since $m(r)\to0$, we may then move
        $R_0$ outward until
        \[
                C_F(C_0+C_1C_H\delta)^2\le C_H.
        \]
        Substitution in \eqref{eq:general-lower-taylor} yields
        \[
                F(D^2\underline u_R^\infty)
                \ge f_{0,R}(\rho)+\varepsilon(\rho)
                \ge f_R.
        \]
        The same smallness choice keeps both Hessians in the fixed compact
        neighborhood of $A$, with eigenvalues in $\Gamma_k$.

        The second integral condition implies
        \[
                T(R_0)
                =\frac1{n-2}\int_{R_0}^\infty r\varepsilon(r)\,dr<\infty,
        \]
        so $T$ is nonnegative, nonincreasing, and tends to zero. Finally,
        \[
                \underline u_R^\infty-\overline u_R^\infty
                =-2\bigl(T(\rho)-T(R)\bigr)
                +C_H\Psi_R(\rho).
        \]
        This expression is nonpositive, vanishes at $\rho=R$, and has the asserted
        lower bound because $\Psi_R\ge\Psi$.
\end{proof}

For the rest of the construction, retain
\[
        \Phi(x):=[\mathcal H(f_0-1)](\rho(x))
\]
and use the convention $T\equiv0$ in the $A$-adapted radial case. Set
\[
        \Phi^-:=\Phi-T(\rho),
        \qquad
        \Phi^+:=\Phi+T(\rho).
\]
Letting the truncation radius tend to infinity at a fixed tail point in
Lemma~\ref{lem:general-far-field-barriers} gives the untruncated inequalities
\[
        F\bigl(D^2(P_c+\Phi^-+C_H\Psi(\rho))\bigr)\ge f,
        \qquad
        F\bigl(D^2(P_c+\Phi^+)\bigr)\le f.
\]
Thus $P_c+\Phi^-+C_H\Psi(\rho)$ and $P_c+\Phi^+$ are lower and upper comparison functions on the tail. Both have Hessian tending to $A$, and their truncated versions agree on the moving outer boundary.

\subsection{Global comparison functions}

We now complete the comparison system on the fixed part of the exterior domain. We first construct a strict lower barrier near the obstacle, then join it to the far-field subsolution, and finally solve a Dirichlet problem on a fixed inner domain to obtain the upper comparison function.

\begin{lemma}[Strict subsolution near the boundary]
        \label{lem:near-boundary-lower-solution}
        After increasing $R_0$, there exists a $C^2$ function
        $\underline u_\partial$ on a fixed neighborhood of
        $E\cap\{\rho\le2R_0\}$ such that
        \[
                \underline u_\partial=\phi\quad\text{on }\partial\Omega,
                \qquad
                \lambda(D^2\underline u_\partial)\in\Gamma_k,
                \qquad
                F(D^2\underline u_\partial)\ge f.
        \]
        Near $\partial\Omega$, the function agrees with an exponential distance
        barrier and the last inequality is strict.
\end{lemma}

\begin{proof}
        The construction uses the two geometric models recorded in
        Subsection~\ref{subsec:fixed-boundary-geometry}. The distance function gives a strict branch on a boundary collar,
        while a large power of the Minkowski functional extends it across the rest of
        the fixed core.

        Choose a tubular collar on which $d_\Omega$ is smooth, and let
        $\widetilde\phi$ be a smooth extension of $\phi$. Set
        \[
                v_\partial:=\widetilde\phi+\exp(\lambda_\partial d_\Omega)-1.
        \]
        In a principal frame, its Hessian is the sum of a bounded matrix, a positive
        tangential curvature term of size $\lambda_\partial\exp(\lambda_\partial d_\Omega)$,
        and a positive normal term of size
        $\lambda_\partial^2\exp(\lambda_\partial d_\Omega)$. Uniform strict $(k-1)$-convexity of the
        parallel hypersurfaces therefore gives, for all sufficiently large $\lambda_\partial$,
        \[
                \lambda(D^2v_\partial)\in\Gamma_k,
                \qquad F(D^2v_\partial)\ge f+2
        \]
        on a smaller fixed collar. The branch also satisfies
        $v_\partial=\phi$ on $\partial\Omega$.

        Choose $S_{\mathrm M}>1$ so that
        $E\cap\{\rho\le2R_0\}$ is contained in the fixed ring
        $\{1\le\beta_\Omega\le S_{\mathrm M}\}$. Select levels
        \[
                1<\sigma_-<\sigma_0<\sigma_+<S_{\mathrm M}
        \]
        inside the exponential collar, extend $\widetilde\phi$ across the ring, and
        set
        \[
                u_{\mathrm M}:=\beta_\Omega^{p_\mathrm M}-1+\widetilde\phi-\zeta_\mathrm M.
        \]
        The exponent $p_\mathrm M$ will force admissibility and strictness, while the constant
        $\zeta_\mathrm M$ will later determine which branch is selected on the two transition
        bands.

        Put
        \[
                q_\Omega:=\nabla\beta_\Omega,
                \qquad B_\Omega:=D^2\beta_\Omega,
        \]
        \[
                \lambda_{\mathrm{tan}}
                :=p_\mathrm M\beta_\Omega^{p_\mathrm M-1},
                \qquad
                \lambda_{\mathrm{nor}}
                :=p_\mathrm M(p_\mathrm M-1)\beta_\Omega^{p_\mathrm M-2}.
        \]
        Then
        \[
                D^2u_{\mathrm M}
                =D^2\widetilde\phi
                +\lambda_{\mathrm{tan}}B_\Omega
                +\lambda_{\mathrm{nor}}q_\Omega\otimes q_\Omega.
        \]
        For the Newton transform $\mathcal T_{j-1}(B_\Omega)$,
        \[
                \begin{aligned}
                         & S_j(\lambda_{\mathrm{tan}}B_\Omega
                        +\lambda_{\mathrm{nor}}q_\Omega\otimes q_\Omega) \\
                         & \qquad=\lambda_{\mathrm{tan}}^jS_j(B_\Omega)
                        +\lambda_{\mathrm{nor}}\lambda_{\mathrm{tan}}^{j-1}
                        q_\Omega^T\mathcal T_{j-1}(B_\Omega)q_\Omega.
                \end{aligned}
        \]
        The level-set geometry from
        Subsection~\ref{subsec:fixed-boundary-geometry} gives
        \[
                q_\Omega^T\mathcal T_{j-1}(B_\Omega)q_\Omega
                =|q_\Omega|^{j+1}\sigma_{j-1}(\kappa_1,\ldots,\kappa_{n-1})
                \ge c_0>0,
                \qquad 1\le j\le k.
        \]
        Since $\lambda_{\mathrm{nor}}/\lambda_{\mathrm{tan}}=(p_\mathrm M-1)/\beta_\Omega$,
        the rank-one term dominates both the
        $\lambda_{\mathrm{tan}}^jS_j(B_\Omega)$ term and the bounded perturbation $D^2\widetilde\phi$ as
        $p_\mathrm M\to\infty$. Thus $\lambda(D^2u_{\mathrm M})\in\Gamma_k$ and
        \[
                S_k(D^2u_{\mathrm M})\ge c\lambda_{\mathrm{nor}}\lambda_{\mathrm{tan}}^{k-1}.
        \]
        When $l\ge1$, the corresponding upper estimate
        $S_l(D^2u_{\mathrm M})\le C\lambda_{\mathrm{nor}}\lambda_{\mathrm{tan}}^{l-1}$ yields
        $F(D^2u_{\mathrm M})\ge c\lambda_{\mathrm{tan}}$. For $l=0$,
        \[
                F(D^2u_{\mathrm M})\ge c\lambda_{\mathrm{nor}}^{1/k}\lambda_{\mathrm{tan}}^{(k-1)/k}
                \ge c\lambda_{\mathrm{tan}}.
        \]
        Hence, after increasing $p_\mathrm M$ once more,
        \[
                F(D^2u_{\mathrm M})\ge f+1
        \]
        on the fixed ring.

        The growth of $\beta_\Omega^{p_\mathrm M}$ gives
        \[
                \sup_{\{1\le\beta_\Omega\le\sigma_-\}}(u_{\mathrm M}-v_\partial)
                <
                \inf_{\{\sigma_0\le\beta_\Omega\le\sigma_+\}}(u_{\mathrm M}-v_\partial).
        \]
        Choose the shift $\zeta_\mathrm M$ so that $v_\partial$ dominates on the first band and
        $u_{\mathrm M}$ dominates on the second, both with a fixed positive margin.

        At this first use, fix a convex $C^2$ regularization
        $\mathcal M_\delta(p,q)$ of $\max\{p,q\}$. On its transition set,
        \begin{equation}\label{eq:regularized-max-hessian}
                D^2\mathcal M_\delta(p,q)
                =\theta D^2p+(1-\theta)D^2q
                +\lambda D(p-q)\otimes D(p-q),
                \qquad 0\le\theta\le1,
                \quad \lambda\ge0.
        \end{equation}
        Convexity of the admissible matrix set, ellipticity, and concavity of $F$ show that this
        operation preserves admissibility and the common lower bound for $F$.
        A localized regularized maximum therefore equals $v_\partial$ near
        $\partial\Omega$, equals $u_{\mathrm M}$ before leaving the fixed ring, and gives
        the required function $\underline u_\partial$.
\end{proof}

\begin{lemma}[Joining the subsolution]
        \label{lem:gluing-lower-solutions}
        There exists $\delta_{\rm g}>0$ such that, for every sufficiently large $c$
        and every $0<\delta_*<\delta_{\rm g}$, there exists a $C^2$ $k$-admissible
        transition branch $\underline v_{\delta_*}$ on $\{\rho\ge R_0\}$ such that
        \[
                F(D^2\underline v_{\delta_*})\ge f.
        \]
        It connects the fixed-boundary and far-field subsolutions in two ways:
        \begin{enumerate}
                \item \emph{Crossing condition for gluing.} Within the fixed region
                      $\{R_0<\rho<2R_0\}$, the transition branch $\underline v_{\delta_*}$ crosses
                      $\underline u_\partial$ from strictly below to strictly above. The
                      vertical separation margins on either side of the crossing are
                      independent of $c$ and $\delta_*$.
                \item \emph{Behavior at infinity.} On a tail depending on $c$,
                      \[
                              \underline v_{\delta_*}-P_c-\Phi^-\longrightarrow-\delta_*
                              \qquad\text{as }\rho\to\infty.
                      \]
        \end{enumerate}
\end{lemma}

\begin{proof}
        The boundary and far-field branches have incompatible additive normalizations:
        the former is fixed by $\phi$, whereas the latter contains the parameter $c$.
        The transition is therefore carried out on two scales. An increasing affine
        function creates a fixed crossing near $R_0$, and a decaying
        $\mathcal L$-harmonic term creates a second crossing at a radius comparable to
        $c$. Since the tail comparison function is not exactly quadratic, a vanishing
        radial potential is added to absorb the resulting Taylor errors.

        For a parameter $c$ to be chosen at the end, write
        \[
                P_0(x):=\frac12x^TAx+b\cdot x,
                \qquad
                \mathcal W:=\Phi^-+C_H\Psi(\rho),
        \]
        so that $P_c=P_0+c$, and fix a positive gap $\delta_2$. By
        Lemmas~\ref{lem:far-field-barriers} and
        \ref{lem:general-far-field-barriers}, with $T\equiv0$ in the $A$-adapted radial
        case, $P_c+\mathcal W$ is a subsolution on a far tail and
        $D^2\mathcal W\to0$. Choose fixed matrix neighborhoods
        $\mathcal U_0\Subset\mathcal U\Subset\operatorname{Sym}(n)$ of $A$, chosen so that every matrix in $\mathcal U$ has eigenvalues in $\Gamma_k$ and
        $A+D^2\mathcal W\in\mathcal U_0$ there. On $\mathcal U$, write the Lipschitz bound
        for the derivative of $F$ as
        \[
                \|DF(X)-DF(Y)\|\le L_F\|X-Y\|.
        \]
        Let $\omega(r)\downarrow0$ be a continuous decreasing majorant of
        $\sup_{\rho(x)\ge r}|D^2\mathcal W(x)|$.

        Choose radii $r_-<r_+$ and
        $\gamma>0$ such that
        \[
                R_0<r_--\gamma<r_-+\gamma
                <r_+-\gamma<r_++\gamma<2R_0,
        \]
        and set
        \[
                K_-:=\{|\rho-r_-|\le\gamma\},
                \qquad
                K_+:=\{|\rho-r_+|\le\gamma\}.
        \]
        Fix also a margin $\delta_1>0$. These choices are independent of $c$.

        \smallskip
        \noindent\emph{Step 1.}
        Choose an increasing affine function $\ell$ and set
        \[
                v:=P_0+\mathcal W+\ell(\rho).
        \]
        Its slope is chosen so that
        \[
                \begin{aligned}
                        \ell(r_+-\gamma)-\ell(r_-+\gamma)
                        \ge{} & \sup_{K_+}(\underline u_\partial-P_0-\mathcal W)             \\
                              & -\inf_{K_-}(\underline u_\partial-P_0-\mathcal W)+4\delta_1.
                \end{aligned}
        \]
        A vertical shift of $\ell$ then enforces
        \[
                \underline u_\partial\ge v+\delta_1\quad\hbox{on }K_-,
                \qquad
                v\ge\underline u_\partial+3\delta_1\quad\hbox{on }K_+.
        \]
        Because $D^2(\ell\circ\rho)=\ell'D^2\rho\ge0$, this affine adjustment also
        preserves the subsolution inequality.

        \smallskip
        \noindent\emph{Step 2.} Write
        \[
                \ell(r)=ar+d,
                \qquad a=\ell'>0.
        \]
        Choose $\kappa>0$ so small that
        \[
                2^{n-2}-1-(2^{n-1}-1)a\kappa>0,
                \qquad 1-2a\kappa>0,
        \]
        and then choose $\sigma>0$ so small that
        \begin{equation}\label{eq:crossing-parameter-choice}
                \begin{aligned}
                        2^{n-2}-1-(2^{n-1}-1)a\kappa-(2^{n-2}+1)\sigma & >0, \\
                        1-2a\kappa-\sigma                              & >0.
                \end{aligned}
        \end{equation}
        Set $R_*:=\kappa c$ and define
        \begin{align*}
                \Theta_- & =(c-\ell(R_*)+4\delta_2+\sigma c)R_*^{n-2},     \\
                \Theta_+ & =(c-\ell(2R_*)-4\delta_2-\sigma c)(2R_*)^{n-2}.
        \end{align*}
        The second inequality in \eqref{eq:crossing-parameter-choice} gives
        $\Theta_+>0$ for all sufficiently large $c$. Also
        $1-a\kappa+\sigma>0$, so $\Theta_->0$ for all sufficiently large $c$.
        Moreover,
        \begin{align*}
                \Theta_+-\Theta_-
                =R_*^{n-2}\Big\{ & c\big[2^{n-2}-1-(2^{n-1}-1)a\kappa
                                            -(2^{n-2}+1)\sigma\big]         \\
                                 & -(2^{n-2}-1)d-4(2^{n-2}+1)\delta_2\Big\},
        \end{align*}
        which is positive for all sufficiently large $c$ by the first inequality in
        \eqref{eq:crossing-parameter-choice}. Thus
        $0\le\Theta_-<\Theta_+$. Fix $\Theta\in[\Theta_-,\Theta_+]$.

        \smallskip
        \noindent\emph{Step 3.}
        Choose $M>0$ first and then $\Lambda>0$ relative to $M$, both independently of $c$. Set
        \[
                \mathfrak p(r):=\exp\left(\Lambda\int_{R_*}^r\frac{\omega(s)}s\,ds\right),
        \]
        \[
                \mathfrak e(r):=M\left[
                        \Theta\omega(r)\mathfrak p(r)r^{-n}
                        +\Theta^2\mathfrak p(r)^2r^{-2n}\right],
        \]
        \[
                z(r):=-\int_r^\infty t^{1-n}
                \int_{R_*}^t s^{n-1}\mathfrak e(s)\,ds\,dt,
                \qquad Z(x):=z(\rho(x)).
        \]
        Fix $0<\vartheta<\min\{n-2,n/2\}$. Since $\omega(r)\to0$, increasing $c$
        gives $\Lambda\omega(r)\le\vartheta$ for $r\ge R_*$, and therefore
        \[
                1\le\mathfrak p(r)\le(r/R_*)^\vartheta.
        \]
        The exterior potential is finite, and direct differentiation, together with
        $\mathfrak p'=\Lambda\omega\mathfrak p/r$, gives
        \begin{equation}\label{eq:gluing-potential-hessian}
                \begin{aligned}
                        \mathcal LZ       & =\mathfrak e(\rho),                           \\
                        |D^2Z|\le CM\big( & \Theta\omega(\rho)\mathfrak p(\rho)\rho^{-n}
                        +\Theta^2\mathfrak p(\rho)^2\rho^{-2n}                            \\
                                          & +\Lambda^{-1}\Theta\mathfrak p(\rho)\rho^{-n}
                        +\Theta^2R_*^{-n}\rho^{-n}\big).
                \end{aligned}
        \end{equation}
        \smallskip
        \noindent\emph{Step 4.}
        Set
        \[
                J(x):=D^2(\rho^{2-n})(x),
                \qquad
                B(x):=A+D^2\mathcal W(x),
                \qquad
                H(x):=-\Theta J(x)+D^2Z(x).
        \]
        Then $\mathcal L(\rho^{2-n})=0$ and $|J(x)|\le C\rho(x)^{-n}$.
        Since $P_c+\mathcal W$ is a subsolution,
        \[
                F(B(x))\ge f(x).
        \]
        Taylor's formula at $B(x)$ and the Lipschitz continuity of $DF$ on
        $\mathcal U$ give
        \begin{align}
                F(B+H)
                 & \ge F(B)+DF(B)[H]-C|H|^2\notag \\
                 & \ge f(x)+DF(A)[H]
                -C\omega(\rho)|H|-C|H|^2.\label{eq:gluing-taylor}
        \end{align}
        Because $\mathcal L(\rho^{2-n})=0$ and $\mathcal LZ=\mathfrak e(\rho)$,
        \[
                DF(A)[H]=\mathfrak e(\rho).
        \]

        We now verify that the last two terms in
        \eqref{eq:gluing-taylor} are absorbed by $\mathfrak e$. Put
        \[
                X(r):=\Theta\mathfrak p(r)r^{-n},
                \qquad
                \eta_*:=\sup_{r\ge R_*}\bigl(\omega(r)+X(r)\bigr).
        \]
        Since $\mathfrak p(r)\le(r/R_*)^\vartheta$ with $\vartheta<n$ and
        $\Theta R_*^{-n}=O(R_*^{-1})$, one has $\eta_*\to0$ as $c\to\infty$.
        Moreover, the last term in \eqref{eq:gluing-potential-hessian} satisfies
        \[
                \Theta^2R_*^{-n}r^{-n}
                \le \Theta R_*^{-n}X(r)\le\eta_*X(r),
        \]
        and hence
        \begin{equation}\label{eq:gluing-potential-hessian-compressed}
                |D^2Z|
                \le CM\Big(\omega X+X^2+(\Lambda^{-1}+\eta_*)X\Big).
        \end{equation}
        Since
        \[
                \mathfrak e=M(\omega X+X^2),
        \]
        we obtain from \eqref{eq:gluing-potential-hessian-compressed}
        \begin{align*}
                C\omega\bigl(\Theta r^{-n}+|D^2Z|\bigr)
                 & \le \frac{C}{M}\mathfrak e
                +C\omega\mathfrak e
                +C(\Lambda^{-1}+\eta_*)\mathfrak e, \\
                C|H|^2
                 & \le \frac{C}{M}\mathfrak e
                +CM\eta_*^2\mathfrak e
                +CM(\Lambda^{-1}+\eta_*)^2\mathfrak e.
        \end{align*}
        Here we used $X\le\eta_*$ and
        $\omega X+X^2=X(\omega+X)\le2\eta_*X$. Choose $M$ first so that the
        terms $C/M$ are small, then choose $\Lambda$ so that the terms containing
        $M\Lambda^{-1}$ are small, and finally increase $c$ so that
        $\eta_*$ and $\omega(R_*)$ are sufficiently small. With these choices,
        \[
                C\omega(\rho)|H|+C|H|^2\le\mathfrak e(\rho).
        \]
        Substitution in \eqref{eq:gluing-taylor} proves
        \[
                F(B+H)\ge f(x).
        \]
        The same smallness choices ensure $B+H\in\mathcal U$. Thus
        \[
                \underline w^\infty
                :=P_c+\mathcal W-\Theta\rho^{2-n}+Z,
                \qquad
                F(D^2\underline w^\infty)\ge f(x)
        \]
        on $\rho\ge R_*$, with the Hessian remaining in $\mathcal U$.

        \smallskip
        \noindent\emph{Step 5.}
        Since $\mathfrak e\ge0$, Tonelli's theorem and
        $\mathfrak p(s)\le(s/R_*)^\vartheta$ give, for $r=R_*,2R_*$,
        \begin{align*}
                -z(r)  & =\frac1{n-2}\left[
                                             r^{2-n}\int_{R_*}^{r}s^{n-1}\mathfrak e(s)\,ds
                                             +\int_r^\infty s\mathfrak e(s)\,ds\right], \\
                |z(r)| & \le C\Theta\omega(R_*)R_*^{2-n}
                +C\Theta^2R_*^{2-2n}\le Cc\omega(R_*)+C=o(c),
        \end{align*}
        where $\Theta\le CcR_*^{n-2}$ and $R_*=\kappa c$.
        Taking $c$ above the finitely many thresholds in Steps 2--5, the definitions
        of $\Theta_-$ and $\Theta_+$ imply
        \[
                \underline w^\infty-v\le-2\delta_2
                \quad\hbox{on }\{\rho=R_*\},
                \qquad
                \underline w^\infty-v\ge2\delta_2
                \quad\hbox{on }\{\rho=2R_*\}.
        \]
        A localized regularized maximum joins $v$ and $\underline w^\infty$ across
        the moving annulus. Denote the resulting lower branch by $\underline w$.
        On its far tail,
        \[
                \underline w-(P_c+\Phi^-)
                =C_H\Psi(\rho)-\Theta\rho^{2-n}+Z\longrightarrow0.
        \]
        Finally choose fixed levels $r_++\gamma<s_0<s_1$, independently of $c$;
        after increasing $c$ once more, $s_1<R_*$. Let $\chi_*(\rho)$ be a fixed
        cutoff which is zero below $s_0$ and one above $s_1$. On the fixed compact set
        $\operatorname{supp}D^2\chi_*$ the branch agrees with the intermediate function
        $v$. Since
        \[
                D^2v=D^2(P_0+\mathcal W)+aD^2\rho,
                \qquad a>0,
        \]
        the lower comparison inequality on the tail, $D^2\rho\ge0$, and uniform ellipticity
        on the fixed compact Hessian range give a strict margin
        \[
                F(D^2v)-f\ge\eta_0>0
        \]
        for some $\eta_0$ independent of $c$. The relevant Hessians form a compact
        matrix set whose eigenvalues lie in a compact subset of $\Gamma_k$. Hence
        there is $\delta_{\rm g}>0$, independent of $c$
        once the preceding thresholds have been fixed, such that every
        $0<\delta_*<\delta_{\rm g}$ preserves both admissibility and the inequality
        $F\ge f$ after subtraction of $\delta_*\chi_*$. Set
        \[
                \underline v_{\delta_*}:=\underline w-\delta_*\chi_*.
        \]
        The fixed crossing inequalities are unchanged, and the far-field difference
        converges to $-\delta_*$. This proves both asserted properties.
\end{proof}

To close the comparison system from above, we solve one exact Dirichlet
problem on a fixed core and compare it with the far-field supersolution.

\begin{lemma}\label{lem:upper-solution}
        There exists a fixed level
        $S_1\in(R_0,2R_0)$ such that, for all sufficiently large $c$, the Dirichlet problem
        \[
                \begin{cases}
                        F(D^2 v) = f   & \text{in } E\cap\{\rho<S_1\}, \\
                        v = \phi       & \text{on } \partial\Omega,    \\
                        v = P_c+\Phi^+ & \text{on } \{\rho=S_1\}
                \end{cases}
        \]
        has a classical $k$-admissible solution $\overline v$. This solution lies
        strictly below $P_c+\Phi^+$ on a one-sided neighborhood of the artificial outer
        boundary. If $\nu$ points toward increasing $\rho$, then
        \[
                \nabla(\overline v-P_c-\Phi^+)\cdot\nu>0
                \quad\text{on }\{\rho=S_1\}.
        \]
\end{lemma}

\begin{proof}
        Choose fixed regular levels
        \[
                R_0<S_0<S_1<2R_0
        \]
        below the inner crossing region supplied by
        Lemma~\ref{lem:gluing-lower-solutions}, so that
        \[
                \mathcal R:=\{x\in E:S_0<\rho(x)<S_1\}
        \]
        is smooth and connected. Set $D_{\mathrm{core}}:=E\cap\{\rho<S_1\}$.

        Enlarge the domain of $\underline u_\partial$, if necessary, so that it is
        defined near $\overline{D_{\mathrm{core}}}$. For $\kappa>0$, define
        \[
                v_\kappa(x):=P_c(x)+[\mathcal H(f_0-1)](S_1)+T(S_1)
                +\frac\kappa2\bigl(\rho(x)^2-S_1^2\bigr).
        \]
        This branch has the desired outer trace and
        $D^2v_\kappa=A+\kappa G^{-1}$. Hence
        \[
                F(A+\kappa G^{-1})=\kappa F(G^{-1}+\kappa^{-1}A)\longrightarrow\infty.
        \]
        First enlarge $c$ so that $P_c+\Phi^+$ lies above
        $\underline u_\partial$ on $\{\rho=S_1\}$. Then choose $\kappa=\kappa(c)$ so that
        $v_\kappa$ is a strict subsolution and lies below $\underline u_\partial$ on a
        fixed inner level. A localized regularized maximum, using
        \eqref{eq:regularized-max-hessian}, gives an admissible subsolution with the
        prescribed data on both boundary components. The bounded-domain subsolution
        theorem in \cite{Guan2023Dirichlet} therefore produces the exact solution
        $\overline v$ on $D_{\mathrm{core}}$.

        Let $\omega_{\mathrm{out}}$ and $h_0$ be the harmonic functions determined by
        \[
                \begin{cases}
                        \Delta\omega_{\mathrm{out}}=0 & \text{in } E\cap\{\rho<S_1\}, \\
                        \omega_{\mathrm{out}}=0       & \text{on } \partial\Omega,    \\
                        \omega_{\mathrm{out}}=1       & \text{on } \{\rho=S_1\},
                \end{cases}
        \]
        and
        \[
                \begin{cases}
                        \Delta h_0=0     & \text{in } E\cap\{\rho<S_1\}, \\
                        h_0=\phi         & \text{on } \partial\Omega,    \\
                        h_0=P_c+\Phi^+-c & \text{on } \{\rho=S_1\}.
                \end{cases}
        \]
        Since admissibility implies
        $\Delta\overline v>0$, harmonic comparison gives
        \[
                \overline v\le h_0+c\omega_{\mathrm{out}}.
        \]
        The strong maximum principle yields
        \[
                \theta:=\sup_{\{\rho=S_0\}}\omega_{\mathrm{out}}<1.
        \]
        The coefficient of $c$ in the right-hand side is therefore strictly smaller
        than its coefficient in $P_c+\Phi^+$. Increasing $c$ gives
        \[
                \overline v<P_c+\Phi^+
                \quad\text{on }\{\rho=S_0\}.
        \]

        On $\mathcal R$, the function $P_c+\Phi^+$ is a supersolution. The
        comparison principle \cite[Appendix~A, Corollary~A.6]{JiangLiLi2022ExteriorQuotient}
        gives
        \[
                \overline v\le P_c+\Phi^+
                \quad\text{on }\overline{\mathcal R}.
        \]
        The strict inequality on $\{\rho=S_0\}$ and the strong maximum principle make
        the inequality strict in the interior. Since the two functions agree on
        $\{\rho=S_1\}$, the Hopf lemma gives the stated normal derivative. In
        particular, the strict ordering holds on a fixed one-sided collar of the
        artificial boundary.
\end{proof}

\subsection{Assembly and solvability on truncated annuli}

Fix from now on one value of $c$ for which the preceding two lemmas hold and,
unless stated otherwise, one gap $0<\delta_*<\delta_{\rm g}$. The comparison branches
$\underline v_{\delta_*}$ and $\overline v$ are therefore fixed; their dependence
on $c$ (and, for the lower transition branch, on $\delta_*$) is suppressed below.
The three lower comparison functions are the boundary barrier $\underline u_\partial$, the
transition branch $\underline v_{\delta_*}$, and the truncated far-field branch
$\underline u_R^\infty$. The first gluing occurs on fixed annuli, while the
second occurs at a scale proportional to $R$. All fixed-scale data are chosen
before $c$, and $c$ is chosen before the exhaustion radius $R$. Thus every
strict crossing margin is independent of $R$.
\begin{proposition}
        \label{prop:truncated-annulus-solvability}
        For every sufficiently large $R$, set
        \[
                D_R:=E\cap\{\rho<R\},
                \qquad
                f_R(x):=1+\chi_R(\rho(x))(f(x)-1).
        \]
        Then there exists a $C^2$ $k$-admissible subsolution $\underline u_R$ with
        \[
                F(D^2\underline u_R)\ge f_R,
                \qquad
                \underline u_R=\phi\text{ on }\partial\Omega,
                \qquad
                \underline u_R=\overline u_R^\infty\text{ on }\{\rho=R\}.
        \]
        Consequently, the truncated problem
        \begin{equation}\label{eq:annular-problem}
                \begin{cases}
                        F(D^2u_R)=f_R            & \text{in }D_R,            \\
                        u_R=\phi                 & \text{on }\partial\Omega, \\
                        u_R=\overline u_R^\infty & \text{on }\{\rho=R\}
                \end{cases}
        \end{equation}
        has a unique classical $k$-admissible solution.
\end{proposition}

\begin{proof}
        The final assembly uses two crossings on different scales. A localized
        regularized maximum first joins the boundary branch to the fixed transition
        branch using only fixed data. A second regularized maximum joins that branch
        to a moving far-field branch whose trace already matches the outer annular
        boundary condition.

        \smallskip
        \noindent\emph{Step 1. Matching the far-field branch}
        Define
        \[
                w_R:=\underline u_R^\infty
                +\kappa\left(\frac{\rho^2}{R^2}-1\right).
        \]
        The added Hessian is $2\kappa R^{-2}G^{-1}\ge0$, so ellipticity preserves
        the subsolution inequality. The added function vanishes on $\{\rho=R\}$,
        and the two far-field branches already have the same trace there. Hence
        $w_R=\overline u_R^\infty$ on the outer boundary.

        \smallskip
        \noindent\emph{Step 2. Creating the crossing.}
        Lemma~\ref{lem:gluing-lower-solutions} gives
        \[
                \underline v_{\delta_*}=P_c+\Phi^- -\delta_*+o(1)
                \qquad\text{as }\rho\to\infty.
        \]
        Choose $t_*<\theta_0$ and set
        $\kappa=\delta_*/(1-t_*^2)$. Then choose
        $t_2<t_*<t_3<\theta_0$. On $\rho=tR$ the truncated radial profile equals its
        uncut profile, while
        \[
                T(R),\qquad \Psi_R(tR),\qquad
                \underline v_{\delta_*}-P_c-\Phi^-+\delta_*
        \]
        tend to zero for every fixed $t<\theta_0$. Thus the leading difference
        between $w_R$ and $\underline v_{\delta_*}$ is
        $\delta_*+\kappa(t^2-1)$, which changes sign at $t=t_*$. A uniform gap
        $\delta>0$ therefore remains for all large $R$:
        \[
                \underline v_{\delta_*}\ge w_R+\delta
                \quad\hbox{near }\rho=t_2R,
                \qquad
                w_R\ge\underline v_{\delta_*}+\delta
                \quad\hbox{near }\rho=t_3R.
        \]

        \smallskip
        \noindent\emph{Step 3. Gluing the fixed and moving transitions.}
        First take a localized regularized maximum of $\underline u_\partial$ and
        $\underline v_{\delta_*}$ across the two fixed crossing regions supplied by
        Lemma~\ref{lem:gluing-lower-solutions}. Then take a second
        regularized maximum with $w_R$ across the two proportional slabs. Formula
        \eqref{eq:regularized-max-hessian} preserves both admissibility and the common
        lower bound for $F$. The strict gaps ensure that the resulting function is
        exactly $\underline u_\partial$ near $\partial\Omega$ and exactly $w_R$
        near $\{\rho=R\}$. It is therefore the required full-boundary subsolution
        $\underline u_R$. The bounded-domain continuity theorem recorded in
        Appendix~\ref{app:annular-solvability} now gives a classical admissible
        solution of \eqref{eq:annular-problem}, and comparison gives uniqueness.
\end{proof}

\section{Uniform estimates on expanding annuli}\label{sec:estimates}

We first collect the annular objects constructed in Section~\ref{sec:construction}.
Let $R_j\to\infty$ and write
\[
        u_j:=u_{R_j},\qquad D_j:=E\cap\{\rho<R_j\},\qquad
        \chi_j(x):=\chi\!\left(\frac{\rho(x)}{R_j}\right),\qquad
        f_j:=1+\chi_j(f-1).
\]
Set $\Phi_j:=\Phi_{R_j}$. The ordered far-field pair is
\begin{equation}\label{eq:standing-far-field-pair}
        \begin{aligned}
                \overline u_j^\infty
                 & =P_c+\Phi_j+T(\rho),                             \\
                \underline u_j^\infty
                 & =P_c+\Phi_j-T(\rho)+2T(R_j)+C_H\Psi_{R_j}(\rho).
        \end{aligned}
\end{equation}
The two branches solve the opposite comparison inequalities for $f_j$, agree
on $\{\rho=R_j\}$, and the annular solution satisfies
\begin{equation}\label{eq:standing-annular-problem}
        \begin{cases}
                F(D^2u_j)=f_j            & \text{in }D_j,            \\
                u_j=\phi                 & \text{on }\partial\Omega, \\
                u_j=\overline u_j^\infty & \text{on }\{\rho=R_j\}.
        \end{cases}
\end{equation}
Equations~\eqref{eq:standing-far-field-pair}--\eqref{eq:standing-annular-problem}
are the standing comparison system for this section. All dependence on the
expanding radius is carried by the cutoff profile and the common outer trace;
the comparison functions on the fixed inner region do not change with $j$.

The estimates are derived in an order that keeps all constants independent of $R_j$. We first establish uniform $C^0$ comparison bounds. We then treat the two boundary components separately: the inner boundary is fixed, while the outer boundary is rescaled to a fixed ellipsoid. Once both boundary Hessians are controlled, the radius-independent boundary-to-interior estimate gives a uniform Hessian bound on the whole annulus.

\subsection{Uniform \texorpdfstring{$C^0$}{C0} comparison estimates}

We begin by transferring the fixed and far-field comparison functions to uniform bounds for the annular solutions.

\begin{lemma}[Uniform comparison bounds]
        \label{lem:uniform-value-collar}
        Let $S_1$ be the fixed outer level from Lemma~\ref{lem:upper-solution}. After
        discarding finitely many indices, the annular solutions satisfy
        \[
                \underline u_\partial\le u_j\le\overline v
                \qquad\text{on }\{\rho\le S_1\},
        \]
        and
        \begin{equation}\label{eq:uniform-tail-comparison}
                \underline u_j^\infty-C(c)
                \le u_j\le\overline u_j^\infty+C(c)
                \qquad\text{on }\{\rho\ge S_1\},
        \end{equation}
        with $C(c)$ independent of $j$. Consequently, wherever $f_j=f$,
        \begin{equation}\label{eq:uniform-value-remainder}
                |u_j-P_c|
                \le C(c)+CH_m(\rho)+C|\Psi(\rho)|+CT(\rho).
        \end{equation}
\end{lemma}

\begin{proof}
        Two comparison arguments yield the estimate. Comparison on the fixed core
        provides a common normalization at a fixed interface, and comparison on the
        outer shell propagates that normalization to the moving boundary.

        On the fixed region, the global annular subsolution agrees with
        $\underline u_\partial$, so comparison gives
        $\underline u_\partial\le u_j$. For the upper bound, combine the exact core
        solution with the far-field supersolution in the following piecewise function:
        \[
                \overline U_j=
                \begin{cases}
                        \overline v,          & \rho\le S_1, \\
                        \overline u_j^\infty, & \rho\ge S_1.
                \end{cases}
        \]
        The traces agree on $\{\rho=S_1\}$. More importantly, the strict Hopf sign
        from Lemma~\ref{lem:upper-solution} prevents an admissible test function from
        touching the interface from below. Hence $\overline U_j$ is a viscosity upper comparison
        function; consequently, $u_j\le\overline U_j$. Thus
        \[
                \underline u_\partial\le u_j\le\overline v
                \quad\hbox{on }\{\rho\le S_1\}.
        \]

        At the interface $\{\rho=S_1\}$, the two far-field branches differ by a
        bounded amount independent of $j$. Enlarging a single constant $C(c)$ gives
        the desired lower and upper inequalities there. Comparison on the remaining
        outer shell propagates these inequalities all the way to $\rho=R_j$, proving
        \eqref{eq:uniform-tail-comparison}. Finally, $|\Phi_j|\le H_m$ and
        \[
                \Psi(\rho)\le\Psi_{R_j}(\rho)\le0.
        \]
        Substituting these bounds, together with $T(R_j)\le T(\rho)$ on the relevant tail, into the comparison estimate yields
        \eqref{eq:uniform-value-remainder}. Because the matching constant is fixed at
        $S_1$, it remains independent of the outer radius.
\end{proof}

\subsection{Boundary gradient and Hessian estimates}

We next estimate the derivatives on the two boundary components. The inner boundary is handled on a fixed collar, while the moving outer boundary is treated after rescaling to a fixed ellipsoid.

\begin{lemma}[$C^1$ estimate on the inner boundary]\label{lem:fixed-collar-gradient}
        There is a fixed exterior collar of $\partial\Omega$ on which
        \[
                |u_j|+|Du_j|\le C(c)
        \]
        uniformly in $j$.
\end{lemma}

\begin{proof}
        On the exponential collar from Lemma~\ref{lem:near-boundary-lower-solution},
        \[
                \underline u_{\partial}
                =\widetilde\phi+\exp(\lambda_\partial d_\Omega)-1
                \le u_j\le \overline v.
        \]
        All three functions have the same trace $\phi$ on $\partial\Omega$. The
        one-sided normal derivatives of the two comparison functions therefore bound
        $\partial_\nu u_j$, while the tangential derivatives are already determined
        by $\phi$. Hence
        \[
                |Du_j|\le C(c)\qquad\hbox{on }\partial\Omega.
        \]

        To extend the estimate into the domain, choose
        $0<\delta'<\delta''<\delta_0$. The artificial face
        $\{d_\Omega=\delta''\}$ lies a fixed positive distance from the true
        boundary $\partial\Omega$. The uniform $C^0$ bound and the local interior gradient estimate
        \cite{Chen2015InteriorGradient} control $Du_j$ near this face. The gradient is
        therefore bounded on both boundary components of the fixed collar
        $\{0<d_\Omega<\delta''\}$. The fixed-domain gradient estimate in
        \cite[Section~2]{Guan2023Dirichlet}, applicable to the degree-one homogeneous
        quotient operator, then gives
        \[
                \sup_{0\le d_\Omega\le\delta'}
                (|u_j|+|Du_j|)\le C(c).
        \]
        The collar estimate follows with a constant independent of $j$.
\end{proof}

\begin{lemma}\label{lem:fixed-boundary-hessian}
        There is $C(c)$, independent of $j$, such that
        \[
                |D^2u_j|\le C(c)
                \qquad\text{on }\partial\Omega.
        \]
\end{lemma}

\begin{proof}
        The boundary Hessian is recovered component by component. The boundary trace
        controls the tangential block, the strict subsolution controls the mixed
        block, and the equation then prevents the remaining normal eigenvalue from
        becoming unbounded.

        \smallskip
        \noindent\emph{Tangential derivatives.}
        Work near a boundary point in coordinates in which
        $\partial\Omega$ is the graph $x_n=\gamma(x')$, with
        $\gamma(0)=|D\gamma(0)|=0$. Differentiating the boundary identity twice
        gives
        \[
                (u_j)_{\alpha\beta}
                =\phi_{\alpha\beta}-(u_j)_\nu\,\mathrm{II}_{\alpha\beta},
                \qquad \alpha,\beta<n.
        \]
        The fixed $C^1$ estimate therefore bounds every tangential--tangential
        component.

        \smallskip
        \noindent\emph{Mixed derivatives.}
        Let
        \[
                \mathcal L_j:=F^{pq}(D^2u_j)\partial_{pq}
        \]
        be the linearized operator, and let $\mathcal T_\alpha$ be a smooth tangential
        vector field which agrees with $\partial_\alpha$ at the chosen boundary
        point and is tangent to $\partial\Omega$. Put
        $w_j:=u_j-\widetilde\phi$. Differentiating the equation tangentially and
        using the fixed $C^1$ bound gives
        \begin{equation}\label{eq:tangential-linearized-bound}
                |\mathcal L_j(\mathcal T_\alpha w_j)|
                \le C\left(1+\sum_iF^{ii}(D^2u_j)\right)
        \end{equation}
        on a fixed boundary half-ball.

        The strict collar subsolution supplies the barrier needed to dominate the
        right-hand side. After reducing the half-ball, the construction of
        $\underline u_\partial$ gives
        \[
                F(D^2\underline u_\partial)\ge f+1.
        \]
        The eigenvalues of its Hessian range there lie in a compact subset of
        $\Gamma_k$, so decreasing
        $D^2\underline u_\partial$ by a sufficiently small fixed multiple of the identity
        still gives $F\ge f+\frac12$. By concavity, followed by
        the standard distance-barrier argument, constants $t>0$, $N>0$, and $c_0>0$,
        independent of $j$, can therefore be chosen so that
        \[
                h_j:=u_j-\underline u_\partial+t d_\Omega-Nd_\Omega^2\ge0
                \quad\hbox{on the boundary of the half-ball},
        \]
        and
        \begin{equation}\label{eq:fixed-boundary-strict-barrier}
                \mathcal L_jh_j
                \le-c_0\left(1+\sum_iF^{ii}(D^2u_j)\right).
        \end{equation}
        For $B$ sufficiently large, the functions
        \[
                Bh_j\pm\mathcal T_\alpha w_j
        \]
        are nonnegative on the boundary of the half-ball and are
        $\mathcal L_j$-superharmonic by
        \eqref{eq:tangential-linearized-bound}--
        \eqref{eq:fixed-boundary-strict-barrier}. The maximum principle and a normal
        derivative evaluation at the boundary point therefore give
        \[
                |(u_j)_{\alpha\nu}|\le C(c)
                \qquad\hbox{on }\partial\Omega.
        \]

        \smallskip
        \noindent\emph{Normal--normal derivative.}
        The preceding estimates bound the tangential block and the mixed vector of
        $D^2u_j$ at the boundary. In an adapted frame, write
        \[
                M(t)=\begin{pmatrix}B&p\\ p^T&t\end{pmatrix}.
        \]
        For every $r\ge1$,
        \[
                S_r(M(t))=tS_{r-1}(B)+R_r(B,p),
        \]
        where $R_r$ is independent of $t$. Thus the original quotient equation
        $S_k(M)=gS_l(M)$ becomes
        \begin{equation}\label{eq:normal-entry-equation}
                t\bigl(S_{k-1}(B)-gS_{l-1}(B)\bigr)
                =gR_l(B,p)-R_k(B,p),
        \end{equation}
        with the convention $S_{-1}=0$ when $l=0$. The standard same-trace strict
        subsolution argument in the boundary Hessian estimate gives the quantitative
        normal-direction nondegeneracy needed to solve
        \eqref{eq:normal-entry-equation} for $t$; see
        \cite[Section~4]{Guan2023Dirichlet} or \cite[Theorem~4.2]{Trudinger95}.
        The fixed charts, the $C^1$ bound,
        the strict-subsolution margin, $\inf f$, and $\|f\|_{C^2}$ are all uniform in
        $j$. Since the tangential and mixed entries are already controlled,
        \[
                \sup_{\partial\Omega}|D^2u_j|\le C(c)
        \]
        with no dependence on $R_j$.
\end{proof}

\begin{proposition}\label{prop:outer-boundary}
        There is a constant $C$, independent of $j$, such that
        \[
                |D^2u_j|\le C
                \qquad\text{on }\{\rho=R_j\}.
        \]
\end{proposition}

\begin{proof}
        After scaling by $R_j$, the moving boundary becomes a fixed ellipsoid. The
        comparison estimate gives a uniform $C^0$ bound on the scaled collar, and a
        fixed-domain boundary estimate gives the required Hessian control.

        Choose $\theta_1<\vartheta_0<\vartheta_1<1$. On the proportional collar
        $\vartheta_0R_j\le\rho\le R_j$, the cutoff has vanished, so
        $F(D^2u_j)=1$. Write the common outer trace as
        $\overline u_j^\infty=P_c+\psi_j(\rho)$ and set
        \[
                c_j:=c+\psi_j(R_j),\qquad
                \widetilde u_j(y):=R_j^{-2}
                \bigl(u_j(R_jy)-b\cdot(R_jy)-c_j\bigr).
        \]
        The normalization removes the constant part of the radial correction and
        gives
        \[
                F(D_y^2\widetilde u_j)=1,
                \qquad
                \widetilde u_j=\frac12y^TAy
                \quad\hbox{on }\{\rho(y)=1\}.
        \]

        For $x=R_jy$, the comparison estimate gives
        \[
                \begin{aligned}
                        |\widetilde u_j(y)-\frac12y^TAy|
                        \le R_j^{-2}\Big( & |\psi_j(\rho(x))-\psi_j(R_j)|
                        +C|\Psi(\rho(x))|+C(c)\Big).
                \end{aligned}
        \]
        The radial term is bounded by $CH_m$, and $T$ is bounded as well. Since $H_m(r)=o(r^2)$ and
        $\Psi(r)=o(r^2)$, the right-hand side tends uniformly to zero on the fixed
        scaled collar. In particular,
        \[
                \|\widetilde u_j\|_{L^\infty(\{\vartheta_0\le\rho\le1\})}\le B_0
        \]
        with $B_0$ independent of $j$.

        We next obtain a scale-independent $C^1$ bound near the outer face. Choose
        $\mu>0$ so that the quadratic function
        \[
                q^-(y):=\frac12y^TAy+\mu(\rho(y)^2-1)
        \]
        satisfies $q^-\le\widetilde u_j$ on $\{\rho=\vartheta_0\}$ for every $j$.
        It agrees with the outer trace on $\{\rho=1\}$ and
        $F(D^2q^-)=F(A+2\mu G^{-1})>1$, so comparison gives
        $q^-\le\widetilde u_j$ throughout the fixed collar.

        For the upper barrier, let $h^+$ solve the fixed harmonic problem
        \[
                \Delta h^+=0\quad\hbox{in }\{\vartheta_0<\rho<1\},
        \]
        \[
                h^+=\frac12y^TAy\quad\hbox{on }\{\rho=1\},
                \qquad
                h^+=B_0+1\quad\hbox{on }\{\rho=\vartheta_0\}.
        \]
        Since $k$-admissibility implies $\Delta\widetilde u_j>0$ and
        $|\widetilde u_j|\le B_0$ on the inner face, the maximum principle gives
        $\widetilde u_j\le h^+$ on the collar. The two barriers have the same trace
        as $\widetilde u_j$ on the outer face, so their fixed normal derivatives
        bound $\partial_\nu\widetilde u_j$ there. Tangential derivatives are fixed
        by the quadratic trace. An interior gradient estimate controls
        $D\widetilde u_j$ on $\{\rho=\vartheta_1\}$. On
        $\{\vartheta_1<\rho<1\}$ the equation has constant right-hand side; after
        differentiating, each directional derivative satisfies the homogeneous
        linearized equation. The maximum principle therefore propagates the two
        boundary gradient bounds across the outer collar. Hence
        \[
                \|\widetilde u_j\|_{C^1(\{\vartheta_1\le\rho\le1\})}\le C
        \]
        with $C$ independent of $j$. The same strict quadratic lower barrier
        supplies the subsolution required by the fixed-domain boundary Hessian estimate
        in \cite[Section~4]{Guan2023Dirichlet}. Tangential derivatives come from
        the fixed quadratic trace and the second fundamental form of the fixed
        ellipsoid, mixed derivatives from the standard tangential barrier, and the
        strict same-trace subsolution together with the equation controls the
        normal--normal entry; see \cite[Section~4]{Guan2023Dirichlet} or
        \cite[Theorem~4.2]{Trudinger95}. Its constant depends only on the
        fixed scaled geometry, $n,k,l,A,G,\vartheta_0,\vartheta_1$, the strictness
        parameter $\mu$, and $B_0$, and is independent of $R_j$. Hence
        \[
                |D_y^2\widetilde u_j|\le C
                \quad\hbox{on }\{\rho(y)=1\}.
        \]
        Since $D_y^2\widetilde u_j(y)=D_x^2u_j(R_jy)$, scaling back gives the asserted estimate
        on $\{\rho=R_j\}$.
\end{proof}

\subsection{Global Hessian and interior H\"older estimates}

With the boundary Hessians uniformly controlled, we now propagate the estimate through the annulus and then apply interior Evans--Krylov estimates on fixed compact subsets.

\begin{lemma}\label{lem:global-hessian}
        There is $C(c)$, independent of $j$, such that
        \[
                \sup_{D_j}|D^2u_j|\le C(c).
        \]
\end{lemma}

\begin{proof}
        The preceding boundary results give
        \[
                \sup_{\partial D_j}|D^2u_j|\le C(c).
        \]
        We also record here the source bounds needed by the global estimate. By definition,
        \[
                f_j=1+\chi_j(f-1).
        \]
        This is a convex combination of $f$ and $1$, so $f_j\ge c_0>0$. On the
        transition region $\rho\asymp R_j$,
        \[
                |D\chi_j|\le CR_j^{-1},
                \qquad |D^2\chi_j|\le CR_j^{-2},
        \]
        while outside that region $\chi_j$ is constant. The global $C^2$ bound for
        $f$ therefore gives
        \begin{equation}\label{eq:cutoff-source-c2-bound}
                f_j\ge c_0,
                \qquad \|f_j\|_{C^2(D_j)}\le C,
        \end{equation}
        with constants independent of $j$.

        Write
        \[
                F(D^2u_j)=f_j,
                \qquad
                \mathcal L_jv:=F^{pq}(D^2u_j)v_{pq}.
        \]
        For a unit vector $\xi$, differentiating twice yields
        \[
                \mathcal L_j(u_j)_{\xi\xi}
                =(f_j)_{\xi\xi}
                -F^{pq,rs}(D^2u_j)
                        (u_j)_{pq\xi}(u_j)_{rs\xi}.
        \]
        Concavity makes the last term nonnegative. For a constant source this would
        already give the desired subharmonicity. Here the source derivatives must be
        retained, and adding a large multiple of $\rho^2$ would introduce boundary
        values of order $R_j^2$, destroying radius independence.

        By the boundary-to-interior estimate \cite[(2.2)]{ITW2005}, which applies
        to the quotient operator in \cite[(1.5)]{ITW2005},
        \[
                \sup_{D_j}|D^2u_j|
                \le \sup_{\partial D_j}|D^2u_j|+C.
        \]
        The constant has no domain-size dependence. The source $C^2$ bound follows
        from \eqref{eq:cutoff-source-c2-bound}, while Proposition~\ref{prop:outer-boundary}
        and Lemma~\ref{lem:fixed-boundary-hessian} uniformly control the two boundary
        Hessians. Hence
        \[
                \sup_{D_j}|D^2u_j|\le C(c).
        \]
\end{proof}

The global Hessian bound gives only a bounded matrix range. Before applying
Evans--Krylov, we record the elementary separation argument that turns the
positive right-hand side into a uniform distance from the boundary of the
G\r{a}rding cone.

\begin{lemma}\label{lem:cone-margin}
        For every $B_0,c_0>0$ there is $\delta>0$ such that
        \[
                \lambda(M)\in\Gamma_k,\qquad |M|\le B_0,\qquad F(M)\ge c_0
        \]
        imply
        \[
                \operatorname{dist}\bigl(\lambda(M),\partial\Gamma_k\bigr)\ge\delta.
        \]
\end{lemma}

\begin{proof}
        Otherwise there is a bounded sequence of symmetric matrices $M_i$ with
        $\lambda(M_i)\in\Gamma_k$ and $F(M_i)\ge c_0$ such that, after passing to a
        subsequence, $\lambda(M_i)$ converges to a finite point of $\partial\Gamma_k$.
        At such a cone boundary point the Maclaurin inequalities force $S_k=0$. If
        $l=0$, then $F(M_i)=S_k(M_i)^{1/k}\to0$. If $l\ge1$ and $S_l(M_i)$ stays
        positive, the ratio $S_k(M_i)/S_l(M_i)$ tends to zero. If $S_l(M_i)\to0$,
        then
        \[
                S_k(M_i)\le C S_l(M_i)^{k/l},
        \]
        so the ratio, and hence $F(M_i)$, again tends to zero. This is a
        contradiction.
\end{proof}

\begin{lemma}\label{lem:local-ek}
        For every compact set $K\Subset E$ there are
        $\alpha_K\in(0,1)$ and $C(K,c)$, independent of $j$, such that
        \[
                \|u_j\|_{C^{2,\alpha_K}(K)}\le C(K,c)
        \]
        for all sufficiently large $j$.
\end{lemma}

\begin{proof}
        Choose $K\Subset K_+\Subset E$. The value and gradient bounds near the
        fixed boundary from Lemma~\ref{lem:fixed-collar-gradient}, together with the global Hessian estimate $\|D^2u_j\|_{L^\infty} \le C$, control
        $\|u_j\|_{L^\infty(K_+)}$ by integration along fixed paths. For all large
        $j$, the cutoff is inactive on $K_+$, so the equation there has the fixed
        right-hand side $f$.

        The global Hessian bound alone gives a bounded matrix range; the positive
        lower bound for $f$ and Lemma~\ref{lem:cone-margin} show that the eigenvalues
        of this range lie in a compact subset of $\Gamma_k$. Thus $F$ is uniformly
        elliptic and concave on the
        entire range of $D^2u_j$ over $K_+$. The interior Evans--Krylov estimate then
        gives the stated $C^{2,\alpha_K}$ bound on $K$ with a constant independent of
        $j$.
\end{proof}

\begin{proposition}[Estimates independent of the truncation radius]
        \label{prop:annular-estimates}
        Fix $c$ sufficiently large for the construction in Section~\ref{sec:construction}. The solutions $u_j$ satisfy the following estimates with
        constants independent of $j$:
        \begin{enumerate}
                \item on a fixed collar of $\partial\Omega$,
                      \[
                              |u_j|+|Du_j|\le C(c);
                      \]
                \item on the whole truncated domain,
                      \[
                              \sup_{\overline D_j}|D^2u_j|\le C(c);
                      \]
                \item on the untruncated part of the tail,
                      \[
                              |u_j-P_c|\le C(c)+CH_m(\rho)+C|\Psi(\rho)|+CT(\rho);
                      \]
                \item for every $K\Subset E$, there are $\alpha_K>0$ and $C(K,c)$ such that
                      \[
                              \|u_j\|_{C^{2,\alpha_K}(K)}\le C(K,c)
                      \]
                      for all sufficiently large $j$.
        \end{enumerate}
\end{proposition}

\section{Exterior solutions, boundary regularity, and asymptotics}
\label{sec:compactness}

The uniform annular estimates now allow us to pass to the exterior domain. We first take the limit and identify its asymptotic Hessian and gradient. We then recover regularity at the fixed boundary, so that the limit has the regularity asserted in the main theorem. Finally, under the additional pointwise assumptions, we derive the higher-order expansion and the prescribed-constant result.

\subsection{Construction of the exterior solution and preservation of admissibility}

We first extract an exterior limit from the annular solutions and record the comparison estimate that survives the limit.

\begin{proposition}\label{prop:compactness}
        After passing to a subsequence, the annular solutions converge to a
        $k$-admissible function
        \[
                u\in C^2(E)\cap C^0(E\cup\partial\Omega).
        \]
        More precisely:
        \begin{enumerate}
                \item for every $K\Subset E$, there is $\alpha_K>0$ such that
                      \[
                              u_j\to u\quad\text{in }C^{2,\beta}(K)
                              \quad\text{for every }0<\beta<\alpha_K;
                      \]
                \item the limit solves the original quotient equation
                      \[
                              S_{k,l}(D^2u)=g\quad\text{in }E,
                              \qquad
                              u=\phi\quad\text{on }\partial\Omega;
                      \]
                      equivalently, $F(D^2u)=f$ in $E$;
                \item on a sufficiently far tail,
                      \begin{equation}\label{eq:tail-value-bound}
                              |u(x)-P_c(x)|
                              \le C(c)+CH_m(\rho(x))+C|\Psi(\rho(x))|+CT(\rho(x)).
                      \end{equation}
        \end{enumerate}
\end{proposition}

\begin{proof}
        Choose a compact exhaustion
        $K_1\Subset K_2\Subset\cdots\Subset E$. On each $K_i$,
        Lemma~\ref{lem:local-ek} gives a uniform $C^{2,\alpha_i}$ bound for all
        sufficiently large $j$. Successive extraction and a diagonal
        Arzel\`a--Ascoli argument produce a subsequence such that
        \[
                u_j\to u\quad\hbox{in }C^{2,\beta_i}(K_i),
                \qquad 0<\beta_i<\alpha_i.
        \]
        This defines a function $u\in C^2(E)$.

        Fix $x\in E$. For all sufficiently large $j$, the point lies in the untruncated region and
        $f_j(x)=f(x)$. The equation can be written in the original variables as
        \[
                S_k(D^2u_j(x))-g(x)S_l(D^2u_j(x))=0,
        \]
        which passes to the limit under local $C^2$ convergence. On each
        compact set, the Hessians are bounded and $F(D^2u_j)=f$ has a positive lower
        bound. Lemma~\ref{lem:cone-margin} therefore keeps them a uniform distance
        from the cone boundary. Hence $\lambda(D^2u)\in\Gamma_k$, and the limiting identity is equivalent to
        $S_{k,l}(D^2u)=g$, or, after taking the homogeneous root,
        $F(D^2u)=f$.

        On a fixed closed collar, Lemma~\ref{lem:fixed-collar-gradient} makes the
        sequence equibounded and equi-Lipschitz. After a further subsequence it
        converges uniformly on the collar closure. The limit agrees with the interior
        limit and equals $\phi$ on $\partial\Omega$, because every $u_j$ has that
        trace. Thus $u$ extends continuously to the inner boundary $\partial\Omega$.

        Finally, a fixed far point belongs to the untruncated region for all
        sufficiently large $j$. The uniform comparison estimate gives
        \[
                |u_j(x)-P_c(x)|
                \le C(c)+CH_m(\rho(x))+C|\Psi(\rho(x))|+CT(\rho(x)).
        \]
        Passing to the limit proves \eqref{eq:tail-value-bound}.
\end{proof}

The value estimate \eqref{eq:tail-value-bound} is subquadratic. We next use a blow-down argument to prove convergence of the Hessian.

\begin{lemma}\label{lem:subquadratic-gradient}
        Let $w=u-P_c$ and let $\Xi\ge0$. If
        \[
                |w(x)|\le\Xi(\rho(x)),\qquad
                \Xi(r)/r^2\to0,
        \]
        and $D^2w$ is bounded on a tail, then
        \[
                |Dw(x)|\le\omega_1(\rho(x))\,\rho(x),
                \qquad \omega_1(r)\to0.
        \]
\end{lemma}

\begin{proof}
        A finite-difference estimate is applied at a scale that balances the
        subquadratic value error $\Xi(\rho)$ against the bounded Hessian remainder $\|D^2w\|_{L^\infty} < \infty$. Set
        \[
                \mu(R):=\sup_{t\ge R}\frac{\Xi(t)}{t^2}\longrightarrow0.
        \]
        For a far point $y$, write $s=\rho(y)$ and choose
        \[
                \delta_s:=\min\{\delta,\sqrt{\mu(s/2)}\},
                \qquad h:=\delta_ss,
        \]
        where $\delta>0$ is fixed so that $y\pm h\xi$ stay in a comparable affine
        annulus for every unit vector $\xi$. The case $\mu(s/2)=0$ is immediate;
        otherwise Taylor's formula and the tail Hessian bound $\|D^2w\|_{L^\infty} < \infty$ give
        \[
                w(y+h\xi)=w(y)+hDw(y)\cdot\xi+O(h^2).
        \]
        Both value terms are $O(\mu(s/2)s^2)$. After division by the chosen scale
        $h$, the two competing errors become
        \[
                \frac{|Dw(y)\cdot\xi|}{s}
                \le C\frac{\mu(s/2)}{\delta_s}+C\delta_s\longrightarrow0.
        \]
        Taking the supremum over $\xi$ proves the claim.
\end{proof}

\begin{lemma}\label{lem:dyadic-holder-source}
        Let $f\in C^1$ on a tail $\{\rho\ge R_0\}$ and suppose that, for some
        nonnegative function $m$,
        \[
                |f(x)-1|\le m(\rho(x)),
                \qquad \widehat m:=m+\mathcal A m,
                \qquad
                \int_{R_0}^{\infty}r\widehat m(r)^2\,dr<\infty.
        \]
        Assume also that $Df$ is bounded on the tail. Then, after increasing $R_0$
        if necessary, for every $0<\sigma\le1/4$,
        \[
                \sup_{R\ge2R_0}R^\sigma
                [f]_{C^\sigma(\{R/2\le\rho\le2R\})}<\infty.
        \]
\end{lemma}

\begin{proof}
        The integrability condition $\int_{R_0}^\infty r \widehat m(r)^2\,dr<\infty$ first implies $f\to1$; interpolation
        with the ordinary Lipschitz bound $\|Df\|_{L^\infty} < \infty$ then gives the scale-invariant H\"older
        modulus. Put $h:=f-1$, $L_f:=\|Dh\|_{L^\infty}$, and
        \[
                \mathfrak M:=\int_{R_0}^{\infty}t\widehat m(t)^2\,dt<\infty.
        \]
        The case $L_f=0$ is immediate. The same Lipschitz-persistence argument with a
        fixed threshold first shows that $|h|\le1$ on a sufficiently far tail.
        Now fix a point $x$ on that tail, set $r:=\rho(x)$ and $\zeta:=|h(x)|$, and write
        $x=rv$ with $\rho(v)=1$. Since the affine unit sphere is compact, $|v|$ is
        uniformly bounded. On an interval $I_r$ centered at $r$ and of length
        comparable to $\zeta/L_f$, the Lipschitz estimate gives $|h(tv)|\ge \zeta/2$.
        After increasing $R_0$, this interval lies in $[r/2,2r]$. Since
        $|h(tv)|\le m(t)$ there,
        \[
                \mathfrak M
                \ge \int_{I_r}t m(t)^2\,dt
                \ge c\,r \zeta^3/L_f.
        \]
        Consequently,
        \[
                |f(x)-1|\le Cr^{-1/3}.
        \]
        For $x,y$ in the same affine annulus $\{R/2\le\rho\le2R\}$, we therefore
        have both
        \[
                |f(x)-f(y)|\le CR^{-1/3},
                \qquad
                |f(x)-f(y)|\le L|x-y|.
        \]
        Using $\min(A,Bt)\le A^{1-\sigma}B^\sigma t^\sigma$ gives
        \[
                R^\sigma\frac{|f(x)-f(y)|}{|x-y|^\sigma}
                \le C R^{\sigma-(1-\sigma)/3}.
        \]
        The exponent is nonpositive for $0<\sigma\le1/4$, which proves the claim.
\end{proof}

\begin{proposition}\label{prop:hessian-tail}
        There is a nonnegative function $\omega_2(r)\to0$ such that
        \[
                |D^2u(x)-A|
                \le\omega_2(|x|_A)
        \]
        for all sufficiently large $|x|_A$. In particular,
        \[
                D^2u(x)\longrightarrow A
                \qquad\text{as }|x|\to\infty.
        \]
\end{proposition}

\begin{proof}
        The global Hessian bound $\|D^2u\|_{L^\infty} < \infty$ passes to the limit. Together with the positive lower bound for $f$, Lemma~\ref{lem:cone-margin} places all far-tail Hessians in one compact matrix set $K_0\subset\operatorname{Sym}(n)$ whose eigenvalues lie in a compact subset of $\Gamma_k$. Let $L_\rho$ be a Euclidean Lipschitz constant for $\rho$ and choose $\theta>0$ with $L_\rho\theta\le1/8$. Then every sufficiently far ball $B(x,\theta\rho(x))$ lies in $E$.

        Suppose the conclusion fails. There are points $x_j$ such that
        \[
                r_j:=\rho(x_j)\to\infty,
                \qquad |D^2u(x_j)-A|\ge\varepsilon_0.
        \]
        Put $w=u-P_c$ and define on $B_\theta$
        \[
                V_j(z):=\frac{w(x_j+r_jz)-w(x_j)-Dw(x_j)\cdot r_jz}{r_j^2}.
        \]
        Then $D_z^2V_j(z)=D_x^2u(x_j+r_jz)-A$. Moreover,
        $\rho(x_j+r_jz)\in[r_j/2,2r_j]$ on this fixed ball. Hence
        \[
                F(A+D^2V_j(z))=f(x_j+r_jz),
        \]
        and the matrix arguments stay in $K_0$. The dyadic H\"older bound from Lemma~\ref{lem:dyadic-holder-source} gives a uniform H\"older norm for the rescaled sources, while the envelope in
        $(\mathrm H_A)$ gives uniform convergence to $F(A)=1$.

        Define
        \[
                \mu(R):=\sup_{t\ge R}
                \frac{C(c)+CH_m(t)+C|\Psi(t)|+CT(t)}{t^2}.
        \]
        Then $\mu(R)\to0$. Lemma~\ref{lem:subquadratic-gradient} gives
        $|Dw(x_j)|/r_j\to0$, and the value estimate \eqref{eq:tail-value-bound} gives, uniformly for
        $|z|\le\theta/2$,
        \[
                \frac{|w(x_j+r_jz)|+|w(x_j)|}{r_j^2}\longrightarrow0.
        \]
        Therefore
        \[
                V_j\longrightarrow0
                \quad\hbox{uniformly on }B_{\theta/2}.
        \]

        On the fixed compact matrix range, the equations are uniformly elliptic and
        concave. Evans--Krylov therefore gives a common H\"older modulus for $D^2V_j$
        near the origin. If $D^2V_j(0)$ did not tend to zero, there would be unit
        vectors $\xi_j$ and $\varepsilon_1>0$ such that
        $|\xi_j^TD^2V_j(0)\xi_j|\ge\varepsilon_1$. The common modulus would then give a
        fixed $\delta>0$ on which this directional second derivative has constant
        sign and magnitude at least $\varepsilon_1/2$. For $0<t<\delta$,
        \[
                V_j(t\xi_j)+V_j(-t\xi_j)-2V_j(0)
                =\int_{-t}^t(t-|s|)
                \xi_j^TD^2V_j(s\xi_j)\xi_j\,ds,
        \]
        whose absolute value is at least $\varepsilon_1t^2/2$. This contradicts the
        uniform convergence $V_j\to0$. Hence
        $D^2V_j(0)=D^2u(x_j)-A\to0$, contrary to the choice of $x_j$.

        It follows that
        \[
                \omega_2(R):=\sup_{|x|_A\ge R}|D^2u(x)-A|
        \]
        tends to zero. Hence the Hessian converges as asserted.
\end{proof}

\begin{proposition}[Identification of the asymptotic gradient]
        \label{prop:gradient-tail}
        The exterior solution constructed above satisfies
        \[
                Du(x)-Ax\longrightarrow b
                \qquad\text{as }|x|\to\infty.
        \]
\end{proposition}

\begin{proof}
        Set
        \[
                \Phi=[\mathcal H(f_0-1)]\circ\rho,
                \qquad v:=u-P_c-\Phi.
        \]
        Passing the comparison estimate to the exterior limit shows that $v$ is bounded on a
        far tail. The radial identities give
        \[
                |D\Phi|\le C\rho\,\mathcal A m(\rho),
                \qquad |D^2\Phi|\le C\widehat m(\rho).
        \]
        By Cauchy--Schwarz and the first integrability condition in $(\mathrm H_A)$,
        for fixed $R$ and $r>R$,
        \[
                r\mathcal A m(r)
                \le o_{r\to\infty}(1)
                +C\left(\int_R^\infty s\widehat m(s)^2\,ds\right)^{1/2};
        \]
        hence $D\Phi\to0$, and also $D^2\Phi\to0$. Together with
        Proposition~\ref{prop:hessian-tail}, this gives $D^2v\to0$.

        Finally, boundedness of $v$ and the elementary finite-difference estimate
        \[
                |Dv(y)|\le \frac{C}{h}
                +Ch\sup_{B(y,h)}|D^2v|
        \]
        imply $Dv(y)\to0$ by choosing $h=h(|y|)\to\infty$ with $h=o(|y|)$ and
        $h\sup_{\rho\ge \rho(y)/2}|D^2v|\to0$. Since
        $Du-Ax-b=Dv+D\Phi$, the result follows.
\end{proof}

\subsection{Regularity at the fixed boundary}

We now recover the regularity of the exterior limit at the fixed boundary. The first step is to extend $F$ from the compact Hessian range of the solution to a globally uniformly elliptic concave operator.

\begin{lemma}\label{lem:bellman-extension}
        Let $\mathcal K\subset\operatorname{Sym}(n)$ be compact, with
        $\lambda(N)\in\Gamma_k$ for every $N\in\mathcal K$, and set
        \[
                F_B(M)=\inf_{N\in\mathcal K}
                \{F(N)+DF(N)[M-N]\},
                \qquad M\in\operatorname{Sym}(n).
        \]
        Then $F_B$ is globally defined, concave, and uniformly elliptic. Moreover,
        $F_B=F$ on $\mathcal K$.
\end{lemma}

\begin{proof}
        Every tangent plane lies above the graph of the concave function $F$.
        For $M\in\mathcal K$, choosing the tangent plane based at $N=M$ gives the
        opposite inequality. Since $DF(N)$ is bounded between two positive multiples
        of the identity on $\mathcal K$, the infimum has the same uniform ellipticity
        bounds.
\end{proof}

\begin{proposition}[Regularity at boundary]
        \label{prop:post-limit-regularity}
        Let $u$ be the exterior limit. Suppose $q\ge4$, $0<\alpha_0<1$,
        $\partial\Omega$ and $\phi$ are $C^{q,\alpha_0}$, and $g$ is
        $C^{q-2,\alpha_0}$ on a fixed exterior collar. Then, on a possibly smaller
        fixed collar,
        \[
                u\in C^{q,\alpha_0}.
        \]
        If the data are smooth, then
        $u\in C^\infty_{\mathrm{loc}}(\overline E)$.
\end{proposition}

\begin{proof}
        Since $g$ stays in a fixed positive range, $f=g^{1/(k-l)}$ has the same local
        $C^{q-2,\alpha_0}$ regularity. In a fixed open collar of $\partial\Omega$, the
        Hessian range of the exterior limit lies in a compact matrix set
        $\mathcal K_0\subset\operatorname{Sym}(n)$ whose eigenvalues lie in a compact
        subset of $\Gamma_k$. Lemma~\ref{lem:bellman-extension} therefore
        extends $F$ to a globally uniformly elliptic concave operator agreeing with
        $F$ on this range. Passing to $v=-u$, the operator is convex. On a fixed smooth
        collar $\mathcal C_\delta$, the true boundary datum is $C^{2,\alpha_0}$ and
        the artificial boundary datum is $C^{2,\alpha_1}$ by interior
        Evans--Krylov. Hence \cite[Theorem~1.4(b)]{SilvestreSirakov2014Boundary}
        gives
        \[
                u\in C^{2,\alpha_b}(\overline{\mathcal C_\delta})
        \]
        for some $0<\alpha_b\le\alpha_0$.

        To recover the data exponent, work in flattened boundary coordinates;
        tangential difference quotients and the boundary Schauder estimate first give
        $C^{3,\alpha_b}$ regularity; the normal--normal derivative is recovered from the
        equation using uniform ellipticity. Thus $D^2u$ is Lipschitz, so the
        linearized coefficients are $C^{\alpha_0}$. Repeating the same boundary
        argument with exponent $\alpha_0$ gives $u\in C^{3,\alpha_0}$, and the
        standard differentiated Schauder induction yields
        \[
                u\in C^{q,\alpha_0}
        \]
        because $\partial\Omega,\phi\in C^{q,\alpha_0}$ and
        $g\in C^{q-2,\alpha_0}$. Smooth data give
        $u\in C^\infty_{\mathrm{loc}}(\overline E)$.
\end{proof}

\begin{proof}[Proof of Theorem~\ref{thm:main}]
        Proposition~\ref{prop:truncated-annulus-solvability} supplies solutions on a
        sequence of expanding annuli, and Proposition~\ref{prop:annular-estimates}
        gives estimates independent of the outer radii. Proposition~\ref{prop:compactness}
        yields a $k$-admissible exterior limit solving the original equation and taking
        the boundary value $\phi$. The envelope $m$ for $f$ introduced in Section~\ref{sec:construction} is
        comparable to the original envelope $m_g$ in $(\mathrm H_A)$. Hence the
        two integrability conditions in $(\mathrm H_A)$ give
        \[
                H_m(r)=o(r),\qquad T(r)\to0,\qquad \Psi(r)\to0.
        \]
        Since $\rho=|x|_A\asymp|x|$, the tail estimate gives
        \[
                u(x)-\frac12x^TAx-b\cdot x=o(|x|).
        \]
        Proposition~\ref{prop:hessian-tail} gives $D^2u\to A$, and
        Proposition~\ref{prop:gradient-tail} gives $Du-Ax\to b$, while
        Proposition~\ref{prop:post-limit-regularity} yields
        $u\in C^{q,\alpha}_{\mathrm{loc}}(\overline E)$.
\end{proof}

\subsection{Higher-order asymptotics at infinity}

This subsection proves Theorem~\ref{thm:asymptotic} and
Corollary~\ref{cor:polynomial-asymptotics}. For the decomposition in
Theorem~\ref{thm:asymptotic}, set
\[
        a(r):=g_0(r)-1,
        \qquad
        e(x):=g(x)-g_0(\rho(x)),
        \qquad \rho=|x|_A.
\]
Then, on a sufficiently far tail,
\begin{equation}\label{eq:section4-tail-decomposition}
        g(x)=1+a(\rho(x))+e(x),
\end{equation}
with
\begin{equation}\label{eq:section4-tail-bounds}
        |a^{(j)}(r)|\le Cr^{-\gamma-j},
        \qquad
        |D^je(x)|\le C|x|^{-\beta-j},
        \qquad 0\le j\le N.
\end{equation}
Here $\gamma>1$, $\beta>2$, and $1\le N\le q-2$; in the $A$-adapted radial case
$e\equiv0$ and $\beta=\infty$.

The proof proceeds in three steps. The tail comparison first bounds the remainder after subtracting the radial correction. Rescaling upgrades this value control to derivative bounds, while the constant-coefficient Newton potential converts source decay into decay of the solution. The quadratic Taylor remainder is then iterated to reach the final rate, after which the correction for $F(D^2u)=f$ is compared with $\Phi_A$.

Set $d:=N-1$. On each annulus of radius $R$, the pointwise bounds through order $N$ and the mean-value theorem give the scaled $C^{d,\alpha}$ bounds used below.

For $R>0$, set
\[
        I_R:=[R/2,2R],
        \qquad
        \mathfrak A_R:=\{x\in E:R/2<\rho(x)<2R\}.
\]
For a one-dimensional interval or a spatial domain $U$ of scale $R$, write
\[
        \|h\|_{C_R^{d,\alpha}(U)}
        :=\sum_{j=0}^d R^j\|D^jh\|_{L^\infty(U)}
        +R^{d+\alpha}[D^dh]_{C^\alpha(U)}.
\]

Starting from the decomposition of $g$ in
\eqref{eq:section4-tail-decomposition}, define the corresponding radial and remainder terms for $f$ by
\begin{equation}\label{eq:weighted-tail-decomposition}
        \widetilde a(r):=(1+a(r))^{1/(k-l)}-1,
        \qquad
        \widetilde e(x):=f(x)-1-\widetilde a(\rho(x)).
\end{equation}
The pointwise bounds \eqref{eq:section4-tail-bounds}, after the
fixed affine change of variables and the preceding mean-value estimate, imply the technical scaled estimates
\begin{equation}\label{eq:weighted-source-bounds}
        \sup_{R\ge R_1}\left(
        R^\gamma\|\widetilde a\|_{C_R^{d,\alpha}(I_R)}
        +R^\beta\|\widetilde e\|_{C_R^{d,\alpha}(\mathfrak A_R)}
        \right)<\infty.
\end{equation}
In the $A$-adapted radial case, $\widetilde e\equiv0$ and $\beta=\infty$. Taylor's
formula for the root map also gives
\begin{equation}\label{eq:radial-source-difference}
        \sup_{R\ge R_1}R^{2\gamma}
        \left\|\widetilde a-\frac{a}{k-l}\right\|_{C_R^{d,\alpha}(I_R)}<\infty.
\end{equation}
For the remainder of the proof, let
\[
        \Phi(x):=[\mathcal H \widetilde a](\rho(x)).
\]
This is the radial correction for $F(D^2u)=f$. Its relation to $\Phi_A$ is used at the end of the proof.

\begin{lemma}
        \label{lem:dyadic-derivative-upgrade}
        Let $w=u-P_c$ and suppose that, on a far tail, the matrices
        $A+D^2w$ remain in a fixed compact matrix set whose eigenvalues lie in a
        compact subset of $\Gamma_k$. Then there are
        $0<\theta<1$ and $C>0$ such that, with $r=\rho(x)$,
        \begin{equation}\label{eq:dyadic-derivative-upgrade}
                \begin{aligned}
                         & \sum_{j=2}^{d+2}r^j
                        \|D^jw\|_{L^\infty(B(x,\theta r/2))}
                        +r^{d+2+\alpha}
                                 [D^{d+2}w]_{C^\alpha(B(x,\theta r/2))} \\
                         & \qquad\le C\left(
                        \|w\|_{L^\infty(B(x,\theta r))}
                        +r^2\|f-1\|_{C_r^{d,\alpha}(B(x,\theta r))}
                        \right)
                \end{aligned}
        \end{equation}
        for every sufficiently far center $x$. The constants are independent of
        $x$.
\end{lemma}

\begin{proof}
        Rescale by
        \[
                W(y):=r^{-2}w(x+ry)
        \]
        on a fixed ball. The equation becomes
        \[
                F(A+D_y^2W)=f(x+ry),
        \]
        and the Hessian arguments remain in one compact matrix set whose eigenvalues
        lie in a compact subset of $\Gamma_k$.
        Thus the rescaled equations are uniformly elliptic and concave with fixed
        structural constants. Evans--Krylov gives the scaled $C^{2,\alpha}$
        estimate. Since $F$ is smooth on the compact Hessian range, differentiating
        the equation and applying the interior linear Schauder estimate inductively
        gives the $C^{d+2,\alpha}$ estimate. Scaling back proves
        \eqref{eq:dyadic-derivative-upgrade}.
\end{proof}

\begin{lemma}\label{lem:constant-coefficient-potential}
        Let $\mathcal L=\operatorname{tr}(GD^2)$ with $G>0$, let $s>2$, and let
        $v$ be bounded on an exterior domain. Suppose that on a far tail,
        $\mathcal Lv=h$ and
        \[
                \sup_{R\ge R_0}R^s
                \|h\|_{C_R^{d,\alpha}(\mathfrak A_R)}<\infty.
        \]
        Then there is a constant $v_\infty$ such that
        \[
                \sup_{R\ge R_1}
                \frac{\|v-v_\infty\|_{C_R^{d+2,\alpha}(\mathfrak A_R)}}{\mathcal P_s(R)}<\infty.
        \]
\end{lemma}

\begin{proof}
        A fixed affine change of variables reduces $\mathcal L$ to the Euclidean
        Laplacian. Since this change preserves dyadic annuli up to fixed constants,
        it is enough to prove the assertion for $\Delta$ and the Euclidean radius.
        Extend $h$ across a fixed ball to a function $\bar h$ whose compact part is
        smooth and whose tail obeys the same weighted $C^{d,\alpha}$ estimate. Define
        its Newton potential by
        \begin{equation}\label{eq:newton-potential-definition}
                \mathcal N\bar h(x)
                :=c_n\int_{\mathbb R^n}|x-y|^{2-n}\bar h(y)\,dy.
        \end{equation}
        The integral is absolutely convergent at infinity because
        \[
                \int_1^\infty t^{2-n}t^{-s}t^{n-1}\,dt
                =\int_1^\infty t^{1-s}\,dt<\infty.
        \]
        Near $y=x$ it is the ordinary Newton potential of a locally
        $C^{d,\alpha}$ function, and hence
        $\Delta\mathcal N\bar h=\bar h$.

        We first record the value estimate for $\mathcal L^{-1}h$. Let $r=|x|$ and split the integral in
        \eqref{eq:newton-potential-definition} into
        \[
                |y|<r/2,
                \qquad r/2<|y|<2r,
                \qquad |y|>2r.
        \]
        On the first region, $|x-y|\asymp r$, so
        \[
                \int_{|y|<r/2}|x-y|^{2-n}|\bar h(y)|\,dy
                \le Cr^{2-n}\left(1+\int_1^r t^{n-1-s}\,dt\right).
        \]
        This is bounded by $C\mathcal P_s(r)$, with the logarithm appearing exactly
        when $s=n$. Scaling $y=rz$ on the middle region gives the bound
        $Cr^{2-s}$, while on the last region
        \[
                \int_{|y|>2r}|x-y|^{2-n}|\bar h(y)|\,dy
                \le C\int_{2r}^\infty t^{1-s}\,dt
                \le Cr^{2-s}.
        \]
        Thus
        \begin{equation}\label{eq:newton-potential-value-bound}
                |\mathcal N\bar h(x)|\le C\mathcal P_s(|x|).
        \end{equation}

        The derivative estimate follows from the same decomposition without
        introducing a singular integral across the observation point. On an annulus
        $\mathfrak A_R$, decompose $\bar h=h_{\rm loc}+h_{\rm rem}$, where
        $h_{\rm loc}$ is supported in
        $\{R/4<|y|<4R\}$ and agrees with $\bar h$ on a neighborhood of
        $\mathfrak A_R$. After scaling to a fixed annulus, the standard interior
        estimate for the Newton potential gives
        \[
                \|\mathcal Nh_{\rm loc}\|_{C_R^{d+2,\alpha}(\mathfrak A_R)}
                \le CR^{2-s}.
        \]
        The function $\mathcal Nh_{\rm rem}$ is harmonic on a fixed enlargement of
        $\mathfrak A_R$. Its supremum there is bounded by the three-region estimate
        above, so the scale-invariant interior estimates for harmonic functions give
        \[
                \|\mathcal Nh_{\rm rem}\|_{C_R^{d+2,\alpha}(\mathfrak A_R)}
                \le C\mathcal P_s(R).
        \]
        Since $R^{2-s}\le C\mathcal P_s(R)$ for every $s>2$, we conclude that
        \begin{equation}\label{eq:newton-potential-scaled-bound}
                \|\mathcal N\bar h\|_{C_R^{d+2,\alpha}(\mathfrak A_R)}
                \le C\mathcal P_s(R).
        \end{equation}

        Now $q:=v-\mathcal N\bar h$ is harmonic on a far exterior region. It is
        bounded because $v$ is bounded and
        \eqref{eq:newton-potential-value-bound} tends to zero. The isolated-singularity
        expansion of the Kelvin transform, equivalently the spherical-harmonic
        expansion of a bounded exterior harmonic function, gives a constant
        $v_\infty$ such that
        \[
                q(x)=v_\infty+O(|x|^{2-n}).
        \]
        Scale-invariant harmonic estimates give the corresponding derivatives and
        H\"older seminorm. Combining this with
        \eqref{eq:newton-potential-scaled-bound}, and using
        $R^{2-n}\le C\mathcal P_s(R)$, proves the claim. Reversing the affine change
        of variables completes the proof for $\mathcal L$.
\end{proof}

\begin{lemma}[Quadratic remainder estimate]
        \label{lem:quadratic-composition}
        Let $\mathcal Q$ be smooth near the origin in $\operatorname{Sym}(n)$ and
        satisfy $\mathcal Q(0)=D\mathcal Q(0)=0$. If, for some $\lambda>0$,
        \begin{equation}\label{eq:hessian-jet-order}
                \sup_{R\ge R_1}R^\lambda
                \|D^2w\|_{C_R^{d,\alpha}(\mathfrak A_R)}<\infty,
        \end{equation}
        then
        \[
                \sup_{R\ge R_1}R^{2\lambda}
                \|\mathcal Q(D^2w)\|_{C_R^{d,\alpha}(\mathfrak A_R)}<\infty.
        \]
\end{lemma}

\begin{proof}
        Every term obtained by differentiating $\mathcal Q(D^2w)$ contains at least
        two factors consisting of derivatives of $D^2w$, because
        $\mathcal Q(0)=D\mathcal Q(0)=0$. The ordinary product and H\"older product
        estimates on each rescaled annulus therefore give the stated order
        $2\lambda$.
\end{proof}

\begin{proof}[Proof of Theorem~\ref{thm:asymptotic}]
        We use the decomposition fixed above. By the zeroth-order bound in \eqref{eq:section4-tail-bounds},
        the construction of Theorem~\ref{thm:main} can use the same radial profile
        $g_0=1+a$ and the remainder bound $\varepsilon(r)=Cr^{-\beta}$ on the tail,
        modified on a compact interval if necessary (and with $\varepsilon\equiv0$
        in the $A$-adapted radial case). Since $\gamma>1$ and $\beta>2$, this choice
        satisfies $(\mathrm H_A)$ and is the splitting used to obtain the solution
        $u$ in the statement.
        Write
        \[
                w:=u-P_c,
                \qquad v:=w-\Phi.
        \]
        The tail comparison gives
        \[
                \Phi-T(\rho)+C_H\Psi(\rho)-C(c)
                \le w\le
                \Phi+T(\rho)+C(c).
        \]
        Since $T$ is bounded and tends to zero, this implies
        \begin{equation}\label{eq:v-bounded-refined}
                |v|\le C(c)
        \end{equation}
        on a far tail.

        The radial identities and the first term in \eqref{eq:weighted-source-bounds} give, for
        $0\le j\le d+2$,
        \begin{equation}\label{eq:phi-weighted-derivatives}
                |D^j\Phi(x)|\le C\rho(x)^{-j}
                \begin{cases}
                        \rho(x)^{2-\gamma},       & 1<\gamma<2,         \\
                        \log\rho(x),              & \gamma=2,\ j=0,     \\
                        1,                        & \gamma=2,\ j\ge1,   \\
                        1,                        & \gamma>2,\ j=0,     \\
                        \rho(x)^{2-\gamma},       & 2<\gamma<n,\ j\ge1, \\
                        \rho(x)^{2-n}\log\rho(x), & \gamma=n,\ j\ge1,   \\
                        \rho(x)^{2-n},            & \gamma>n,\ j\ge1.
                \end{cases}
        \end{equation}
        In the last three ranges the finite limit of $\Phi$ is retained in the
        zeroth-order constant.

        Taylor expansion at $A$ gives
        \[
                F(A+H)=1+DF(A)[H]+\mathcal Q(H),
                \qquad \mathcal Q(0)=D\mathcal Q(0)=0,
        \]
        for all sufficiently far points. Since $\mathcal L\Phi=\widetilde a(\rho)$,
        \begin{equation}\label{eq:refined-v-equation}
                \mathcal Lv=\widetilde e-\mathcal Q(D^2w).
        \end{equation}

        We improve the Hessian-jet decay through a finite iterative loop.

        The boundedness of $v$, \eqref{eq:phi-weighted-derivatives}, the weighted
        source assumptions, and Lemma~\ref{lem:dyadic-derivative-upgrade} imply
        \eqref{eq:hessian-jet-order} with
        \[
                \lambda_0=\min\{\gamma,2\}
        \]
        when $\gamma\ne2$. At $\gamma=2$, the same argument gives every
        $\lambda_0<2$; the endpoint is recovered after the first update.

        Assume that \eqref{eq:hessian-jet-order} holds with exponent $\lambda$.
        Lemma~\ref{lem:quadratic-composition} gives
        \[
                \sup_{R\ge R_1}R^{2\lambda}
                \|\mathcal Q(D^2w)\|_{C_R^{d,\alpha}(\mathfrak A_R)}<\infty,
        \]
        while the remainder $\widetilde e$ has weighted order $\beta$. Thus the right-hand
        side of \eqref{eq:refined-v-equation} has effective order
        \[
                s=\min\{\beta,2\lambda\},
        \]
        with $s=2\lambda$ in the $A$-adapted radial case.

        Lemma~\ref{lem:constant-coefficient-potential} gives a constant $v_\infty$
        and the potential scale corresponding to $s$. Combining this estimate for
        $v-v_\infty$ with \eqref{eq:phi-weighted-derivatives} upgrades the full
        Hessian-jet exponent to
        \[
                \lambda^+=\min\{\gamma,n,\beta,2\lambda\},
        \]
        with $\beta$ omitted in the $A$-adapted radial case.

        At $\gamma=2$, choose the initial exponent $\lambda_0$ with
        $1<\lambda_0<2$. Then $2\lambda_0>2$, while $n\ge3$ and $\beta>2$; hence the
        first update gives $\lambda^+=2$. The explicit endpoint estimate
        $|D^2\Phi|\le Cr^{-2}$ then supplies the exact Hessian order, so no loss is
        retained at $\gamma=2$.

        At $\gamma=n$, the radial estimate gives
        \[
                |D^2\Phi(x)|\le C\rho(x)^{-n}\log\rho(x)
                \le C_\epsilon\rho(x)^{-n+\epsilon}
        \]
        for every $\epsilon>0$. Choose $0<\epsilon<n/2$. The quadratic source then
        has order $2n-2\epsilon>n$. If $\beta\le n$, the remainder term determines
        the potential scale $\mathcal P_\beta$; if $\beta>n$, both source orders lie
        strictly above $n$ and hence produce the capacity scale $r^{2-n}$. Thus the
        small Hessian-exponent loss at $\gamma=n$ does not introduce an additional
        value logarithm. The only endpoint logarithms are those already displayed
        in \eqref{eq:potential-scale} and \eqref{eq:phi-weighted-derivatives}.

        Define
        \[
                \lambda_{j+1}
                :=\min\{\gamma,n,\beta,2\lambda_j\},
        \]
        again omitting $\beta$ in the $A$-adapted radial case. Below the limiting
        threshold the exponent doubles, so the sequence reaches its stable value
        after finitely many iterations. At the stable exponent, the source in
        \eqref{eq:refined-v-equation} has effective order

        \[
                s_{\mathrm{eff}}=
                \begin{cases}
                        \min\{\beta,2\min(\gamma,n)\}, & \widetilde e\not\equiv0, \\
                        2\min(\gamma,n),               & \widetilde e\equiv0,
                \end{cases}
        \]
        up to an arbitrarily small loss when $\gamma=n$. If
        $s_{\mathrm{eff}}\ne s_*$, then both exponents are larger than $n$, and
        hence $\mathcal P_{s_{\mathrm{eff}}}=\mathcal P_{s_*}=r^{2-n}$.
        Consequently Lemma~\ref{lem:constant-coefficient-potential} gives a constant
        $v_\infty$ such that
        \[
                \sup_{R\ge R_1}
                \frac{\|v-v_\infty\|_{C_R^{d+2,\alpha}(\mathfrak A_R)}}
                {\mathcal P_{s_*}(R)}<\infty.
        \]

        It remains to compare $\Phi$ with the correction $\Phi_A$ defined in
        \eqref{eq:intro-linear-particular}. By
        \eqref{eq:radial-source-difference}, the radial source of
        $\Phi-\Phi_A$ has weighted order $2\gamma>2$. The one-dimensional potential
        estimate, equivalently Lemma~\ref{lem:constant-coefficient-potential} after the
        affine change of variables, therefore gives a constant $C_A$ such that
        \[
                \sup_{R\ge R_1}
                \frac{\|\Phi-\Phi_A-C_A\|_{C_R^{d+2,\alpha}(\mathfrak A_R)}}
                {\mathcal P_{2\gamma}(R)}<\infty.
        \]
        Set
        \[
                c_\infty:=c+v_\infty+C_A,
                \qquad
                \tau:=u-\frac12x^TAx-b\cdot x-c_\infty.
        \]
        Then
        \[
                \tau-\Phi_A=(v-v_\infty)+(\Phi-\Phi_A-C_A).
        \]
        Since $s_*=\min\{\beta,2\gamma\}$ and $\rho=|x|_A\asymp|x|$, the last
        two scaled estimates give, for every $0\le j\le N+1$,
        \[
                |D^j(\tau-\Phi_A)(x)|
                \le C|x|^{-j}\mathcal P_{s_*}(|x|)
        \]
        on a sufficiently far tail. This proves \eqref{eq:refined-derivative-decay}.
\end{proof}

\begin{proof}[Proof of Corollary~\ref{cor:polynomial-asymptotics}]
        Use throughout the pointwise splitting from Theorem~\ref{thm:asymptotic}.
        Since $\gamma>2$, the radial correction
        \[
                \Phi=[\mathcal H\widetilde a]\circ\rho
        \]
        has a finite limit; write
        \[
                L_\Phi:=\lim_{|x|\to\infty}\Phi(x).
        \]
        Run the construction with
        \[
                t:=c-L_\Phi.
        \]
        For $c$ sufficiently large, $t$ is above all required thresholds. Let $u_R$ be the unique solution of the truncated problem in
        Proposition~\ref{prop:truncated-annulus-solvability}, with $P_t$ in place of
        $P_c$, and let $u_t$ be an exterior limit supplied by
        Proposition~\ref{prop:compactness}.

        For fixed $t$, the truncated Dirichlet problem is independent of the gluing
        gap $\delta_*$. By uniqueness, its solution $u_R$ may therefore be compared
        with $\underline v_{\delta_*}=\underline w-\delta_*\chi_*$ for every
        sufficiently small $\delta_*$. Passing first to the exterior limit and then
        letting $\delta_*\downarrow0$ gives $u_t\ge\underline w$. The upper
        comparison constructed in Lemma~\ref{lem:uniform-value-collar} gives
        $u_t\le P_t+\Phi^+$. Since
        $\underline w-(P_t+\Phi^-)\to0$, $\Phi^\pm=\Phi\pm T$, and $T\to0$,
        the two inequalities imply
        \[
                u_t-P_t-\Phi\longrightarrow0.
        \]
        Because $t+L_\Phi=c$,
        \[
                u_t(x)-\frac12x^TAx-b\cdot x-c\longrightarrow0.
        \]
        Set $u_c:=u_t$.

        In the notation of the preceding proof, $v_\infty=0$. Since
        $\Phi_A\to0$ under the present normalization and
        $\Phi-\Phi_A-C_A\to0$, we have $C_A=L_\Phi$. Hence
        $c_\infty=t+C_A=c$. The same estimates give,
        with $\mu=\min\{\gamma,\beta\}$,
        \[
                |D^j\tau_c(x)|\le C|x|^{-j}\mathcal P_\mu(|x|),
                \qquad 0\le j\le N+1,
        \]
        where $\tau_c=u_c-\frac12x^TAx-b\cdot x-c$. Since $\Phi_A$ has order
        $\mathcal P_\gamma$ while $\tau_c-\Phi_A$ has order
        $\mathcal P_{s_*}$, one has
        $\mu=\min\{\gamma,s_*\}=\min\{\gamma,\beta\}$. This is
        \eqref{eq:power-tail-consequences}.

        Finally, choose $t_0$ so that the construction works for every $t>t_0$ and
        set $c_*:=t_0+L_\Phi$. Then every prescribed $c>c_*$ is covered. For uniqueness,
        let $u_1,u_2$ have the same boundary value and the same prescribed quadratic
        polynomial, and set
        \[
                \eta_R:=\sup_{\{\rho=R\}}|u_1-u_2|\longrightarrow0.
        \]
        Since adding a constant does not change the equation, comparison on
        $E\cap\{\rho<R\}$ gives $u_1\le u_2+\eta_R$. Letting $R\to\infty$
        yields $u_1\le u_2$, and reversing the roles of the two solutions gives
        $u_1=u_2$.
\end{proof}

\appendix

\section{Angular instability of the linear asymptotics}
\label{app:angular-instability}

We give a Hessian quotient example showing why the radial size and the angular
oscillation of the source enter $(\mathrm H_A)$ differently. Set
\[
        a:=\left(\frac{\binom nl}{\binom nk}\right)^{1/(k-l)},
        \qquad A:=aI,
\]
so that $S_{k,l}(A)=1$. A direct differentiation gives
\begin{equation}\label{eq:isotropic-quotient-linearization}
        DS_{k,l}(A)[H]=\frac{k-l}{na}\operatorname{tr}H.
\end{equation}
For $r=|x|>e^2$, define
\[
        P_A(x):=\frac a2|x|^2,
\]
and
\begin{equation}\label{eq:radial-angular-model-solutions}
        u_{\mathrm{rad}}(x)
        :=P_A(x)+\frac{na}{(k-l)(n-1)}\frac r{\log r},
        \qquad
        u_{\mathrm{ang}}(x)
        :=P_A(x)+\frac{a}{k-l}x_1\log\log r.
\end{equation}
Let
\[
        g_{\mathrm{rad}}:=S_{k,l}(D^2u_{\mathrm{rad}}),
        \qquad
        g_{\mathrm{ang}}:=S_{k,l}(D^2u_{\mathrm{ang}}).
\]

\begin{proposition}[Radial tail versus first angular mode]
        \label{prop:angular-instability}
        For all sufficiently large $r$, both functions in
        \eqref{eq:radial-angular-model-solutions} are $k$-admissible and
        \[
                D^2u_{\mathrm{rad}}\longrightarrow A,
                \qquad
                D^2u_{\mathrm{ang}}\longrightarrow A.
        \]
        Their sources satisfy
        \begin{align}
                g_{\mathrm{rad}}(x)
                 & =1+\frac1{r\log r}
                +O\!\left(\frac1{r(\log r)^2}\right),
                \label{eq:radial-model-source} \\
                g_{\mathrm{ang}}(x)
                 & =1+\frac{x_1}{r^2\log r}
                +O\!\left(\frac1{r(\log r)^2}\right).
                \label{eq:angular-model-source}
        \end{align}
        Hence both obey the same pointwise estimate
        \[
                |g(x)-1|=O((r\log r)^{-1}).
        \]
        The radial source satisfies $(\mathrm H_A)$ after a smooth modification on a
        compact set, and
        \[
                u_{\mathrm{rad}}-P_A=o(r),
                \qquad
                Du_{\mathrm{rad}}-Ax\longrightarrow0.
        \]
        In contrast, the angular source cannot satisfy $(\mathrm H_A)$, while
        \[
                u_{\mathrm{ang}}-P_A\ne o(r),
                \qquad
                Du_{\mathrm{ang}}-Ax\ \text{does not converge}.
        \]
        Thus the second condition in $(\mathrm H_A)$ is an angular stability condition
        for the prescribed linear asymptotics, rather than a consequence of the size of
        $g-1$ alone.
\end{proposition}

\begin{proof}
        Since $D^2(u_{\mathrm{rad}}-P_A)$ and
        $D^2(u_{\mathrm{ang}}-P_A)$ are both
        $O((r\log r)^{-1})$, the two Hessians stay in a fixed compact subset of
        $\Gamma_k$ for all sufficiently large $r$ and converge to $A$.
        Taylor expansion at $A$, using
        \eqref{eq:isotropic-quotient-linearization}, gives
        \[
                S_{k,l}(A+H)
                =1+\frac{k-l}{na}\operatorname{tr}H+O(|H|^2).
        \]
        The elementary identities
        \[
                \Delta\!\left(\frac r{\log r}\right)
                =\frac{n-1}{r\log r}
                +O\!\left(\frac1{r(\log r)^2}\right)
        \]
        and
        \[
                \Delta(x_1\log\log r)
                =\frac{n x_1}{r^2\log r}
                -\frac{x_1}{r^2(\log r)^2}
        \]
        yield \eqref{eq:radial-model-source}--\eqref{eq:angular-model-source}.

        For the radial example, \eqref{eq:isotropic-quotient-linearization} gives
        $G_A=(na)^{-1}I$, hence $\rho=|x|_A=\sqrt{na}\,r$. Define the radial profile by
        $g_0(\sqrt{na}\,r):=g_{\mathrm{rad}}(x)$ for $|x|=r$ and take
        $\varepsilon\equiv0$. After this fixed rescaling,
        \[
                \widehat m_g(\rho)\le \frac{C}{\rho\log \rho}
        \]
        for large $\rho$, and therefore
        \[
                \int^\infty \rho\,\widehat m_g(\rho)^2\,d\rho<\infty.
        \]
        The stated asymptotics follow directly from
        \eqref{eq:radial-angular-model-solutions}.

        For the angular example, the oscillation of $g_{\mathrm{ang}}$ on every large
        sphere satisfies, by \eqref{eq:angular-model-source},
        \[
                \operatorname{osc}_{|x|=r}g_{\mathrm{ang}}
                \ge \frac{c}{r\log r}.
        \]
        Since the level sets of $|x|_A$ are the same spheres up to a fixed rescaling,
        any radial function $g_0(|x|_A)$ satisfying
        \eqref{eq:g-tail-decomposition} must therefore obey
        \[
                \varepsilon(r)\ge \frac{c}{r\log r}
        \]
        after adjusting the constant and the radial variable by this fixed rescaling.
        Hence
        \[
                \int^\infty r\varepsilon(r)\,dr=\infty,
        \]
        so $(\mathrm H_A)$ fails. Finally, along $x=re_1$,
        \[
                \frac{u_{\mathrm{ang}}(re_1)-P_A(re_1)}r
                =\frac{a}{k-l}\log\log r\longrightarrow\infty,
        \]
        while
        \[
                Du_{\mathrm{ang}}(x)-Ax
                =\frac{a}{k-l}e_1\log\log r
                +O((\log r)^{-1}),
        \]
        which does not converge.
\end{proof}

\section{Decay of radial potentials for regularly varying tails}
\label{app:regular-variation}

We record the asymptotics of $\mathcal A a$ and $\mathcal H a$ for regularly varying tails and the thresholds relevant to the integral conditions used above.

Let $a$ be eventually one sign and suppose
\[
        a(r)\sim c_a r^{-\beta}L(r),
        \qquad c_a\ne0,
        \qquad \beta>0,
\]
where $L$ is positive and slowly varying. Write
\[
        \mathcal A a(r):=r^{-n}\int_{R_0}^r s^{n-1}a(s)\,ds,
        \qquad
        \mathcal H a(r):=\int_{R_0}^r t\,\mathcal A a(t)\,dt,
\]
and, when needed,
\[
        \Lambda_j(r):=\int_{R_0}^r\frac{L(t)}t\,dt
        \quad(j=2,n),
        \qquad
        M_\infty:=\int_{R_0}^\infty s^{n-1}a(s)\,ds,
\]
and $(\mathcal H a)_\infty:=\lim_{r\to\infty}\mathcal H a(r)$ when the limit is finite.

\begin{proposition}[Regularly varying radial tails]
        \label{prop:regularly-varying-radial}
        Up to an additive constant in $\mathcal H a$, the asymptotic regimes
        are
        \[
                \begin{array}{c|c|c}
                        \text{range} & \mathcal A a(r) & \mathcal H a(r) \\ \hline
                        0<\beta<2
                                     &
                        \dfrac{c_a}{n-\beta}r^{-\beta}L(r)
                                     &
                        \mathcal H a(r)\sim
                        \dfrac{c_a}{(n-\beta)(2-\beta)}r^{2-\beta}L(r)
                        \\[1.2em]
                        \beta=2
                                     &
                        \dfrac{c_a}{n-2}r^{-2}L(r)
                                     &
                        (\mathcal H a)'(r)\sim\dfrac{c_a}{n-2}\dfrac{L(r)}r
                        \\[1.2em]
                        2<\beta<n
                                     &
                        \dfrac{c_a}{n-\beta}r^{-\beta}L(r)
                                     &
                        (\mathcal H a)_\infty-\mathcal H a(r)\sim
                        \dfrac{c_a}{(n-\beta)(\beta-2)}r^{2-\beta}L(r)
                        \\[1.2em]
                        \beta=n,\ \Lambda_n(r)\to\infty
                                     &
                        c_ar^{-n}\Lambda_n(r)
                                     &
                        (\mathcal H a)_\infty-\mathcal H a(r)\sim
                        \dfrac{c_a}{n-2}r^{2-n}\Lambda_n(r)
                        \\[1.2em]
                        \beta\ge n,\ M_\infty\ne0
                                     &
                        M_\infty r^{-n}
                                     &
                        (\mathcal H a)_\infty-\mathcal H a(r)\sim
                        \dfrac{M_\infty}{n-2}r^{2-n}.
                \end{array}
        \]
        At the critical exponent $\beta=2$, if $\Lambda_2(r)\to\infty$, then
        \[
                \mathcal H a(r)\sim\frac{c_a}{n-2}\Lambda_2(r),
                \qquad
                \mathcal H a(\lambda r)-\mathcal H a(r)
                \sim\frac{c_a}{n-2}L(r)\log\lambda
        \]
        for every fixed $\lambda>0$. If $\Lambda_2$ converges, then $\mathcal H a$
        has a finite limit and
        \[
                (\mathcal H a)_\infty-\mathcal H a(r)
                \sim\frac{c_a}{n-2}\int_r^\infty\frac{L(t)}t\,dt
        \]
        whenever the tail integral is nonzero. At $\beta=n$, if $\Lambda_n$ remains
        bounded, the finite-moment row applies when $M_\infty\ne0$; a vanishing
        accumulated moment requires a finer expansion.
\end{proposition}

\begin{proof}
        The table is obtained by applying Karamata theory twice: first to the weighted
        primitive defining $\mathcal A a$, and then to
        $(\mathcal H a)'=r\mathcal A a$. For $0<\beta<n$, the direct integral
        theorem gives the stated formula for $\mathcal A a$. Applying the direct or
        tail version to its product with $r$ yields the growing regime
        $\beta<2$ and the decaying regime $2<\beta<n$; see
        \cite[Theorem~1.5.11 and Propositions~1.5.9a--b]{BinghamGoldieTeugels}.

        The exponent $\beta=2$ is the logarithmic transition, because
        $(\mathcal H a)'$ is then proportional to $L(r)/r$. The proportional
        increment follows from the local-uniformity theorem for the de Haan
        $\Pi$-class \cite[Theorem~3.1.16]{BinghamGoldieTeugels}. At $\beta=n$, the
        weighted primitive accumulates the logarithmic factor $\Lambda_n$ when it
        diverges. If that primitive converges--and likewise when $\beta>n$--one
        instead obtains a finite moment $M_\infty$; division by $r^n$ followed by one
        integration gives the universal capacity tail $r^{2-n}$. These alternatives
        are exactly the rows displayed in the proposition.
\end{proof}

\begin{remark}[Power-logarithmic tails]
        To illustrate the transitions in Proposition~\ref{prop:regularly-varying-radial}, consider the concrete tail $a(r)=r^{-\beta}(\log r)^p$ for large $r$. Here $L(r)=(\log r)^p$ is slowly varying. The function $\mathcal H a$ has the following asymptotic behaviors:
        \begin{itemize}
                \item When $0<\beta<2$, $\mathcal H a(r)\asymp r^{2-\beta}(\log r)^p$ diverges to infinity.
                \item When $\beta=2$ and $p>-1$, the integral $\Lambda_2$ diverges and $\mathcal H a(r)\asymp (\log r)^{p+1}$ grows logarithmically.
                \item When $2<\beta<n$, $\mathcal H a(r)$ converges to a finite limit $(\mathcal H a)_\infty$, and the remainder decays as $(\mathcal H a)_\infty-\mathcal H a(r)\asymp r^{2-\beta}(\log r)^p$.
                \item When $\beta=n$ and $p>-1$, the remainder decay rate shifts to $r^{2-n}(\log r)^{p+1}$ due to the accumulated mass in $\mathcal A a$.
                \item When $\beta>n$ and the total moment $M_\infty\ne0$, the remainder has the Newtonian capacity rate $r^{2-n}$, regardless of the exponent $p$. If $M_\infty=0$, this leading capacity term cancels and faster decay may occur.
        \end{itemize}
\end{remark}

\begin{remark}[Admissible radial tails]
        If $\widehat m(r)\asymp r^{-\beta}L(r)$ with $L$ slowly varying, then
        \[
                \int^\infty r\widehat m(r)^2\,dr<\infty
        \]
        holds for $\beta>1$, fails for $\beta<1$, and, when $\beta=1$, is equivalent
        to
        \[
                \int^\infty\frac{L(r)^2}{r}\,dr<\infty.
        \]
        If $f_0(r)-1\sim c_fr^{-\beta}L(r)$ and is eventually of one sign, the
        correction $\Phi=\mathcal H(f_0-1)\circ\rho$ follows directly from
        Proposition~\ref{prop:regularly-varying-radial}. If, separately,
        $m(r)\asymp r^{-\beta}L(r)$, then $H_m$ is bounded for
        $\beta>2$, unbounded for $\beta<2$, and at $\beta=2$ is bounded exactly
        when
        \[
                \int^\infty\frac{L(r)}r\,dr<\infty.
        \]
\end{remark}

\section{Solvability on bounded annuli}\label{app:annular-solvability}

This appendix records the bounded-annulus Dirichlet theorem used in the exhaustion, with independent data on the two boundary components.

\begin{proposition}[Dirichlet problem on a bounded annulus]
        Let $q\ge4$ and $0<\alpha<1$. Let
        $\Omega_-\Subset\Omega_+\subset\mathbb R^n$ be bounded
        $C^{q,\alpha}$ domains such that
        \[
                D:=\Omega_+\setminus\overline{\Omega_-}
        \]
        is connected. Write
        \[
                \Gamma_-:=\partial\Omega_-,
                \qquad
                \Gamma_+:=\partial\Omega_+.
        \]
        Let $h\in C^{q-2,\alpha}(\overline D)$ be positive, and prescribe
        independent boundary values
        \[
                \varphi_-\in C^{q,\alpha}(\Gamma_-),
                \qquad
                \varphi_+\in C^{q,\alpha}(\Gamma_+).
        \]
        Suppose there is a full-boundary admissible subsolution
        $\underline u\in C^2(\overline D)$ satisfying
        \[
                \begin{cases}
                        F(D^2\underline u)\ge h & \text{in }D,        \\
                        \underline u=\varphi_-  & \text{on }\Gamma_-, \\
                        \underline u=\varphi_+  & \text{on }\Gamma_+.
                \end{cases}
        \]
        Then the annular Dirichlet problem
        \[
                \begin{cases}
                        F(D^2u)=h   & \text{in }D,        \\
                        u=\varphi_- & \text{on }\Gamma_-, \\
                        u=\varphi_+ & \text{on }\Gamma_+
                \end{cases}
        \]
        has a unique classical $k$-admissible solution
        $u\in C^{q,\alpha_D}(\overline D)$ for some
        $0<\alpha_D\le\alpha$. Smooth data give a smooth solution.
\end{proposition}

\begin{proof}
        Define a single boundary function $\varphi$ on the disconnected boundary
        $\partial D=\Gamma_-\cup\Gamma_+$ by
        \[
                \varphi=\varphi_-\quad\text{on }\Gamma_-,
                \qquad
                \varphi=\varphi_+\quad\text{on }\Gamma_+.
        \]
        Since the two components are disjoint,
        $\varphi\in C^{q,\alpha}(\partial D)$. Thus this is the usual bounded-domain
        Dirichlet problem with separately prescribed data on its two boundary
        components.

        It remains to verify that the bounded-domain subsolution theorem applies to
        the operator $F$. The cone $\Gamma_k\subset\mathbb R^n$ is open,
        symmetric, and convex, while the quotient root is positive, elliptic,
        concave, and homogeneous of degree one on $\Gamma_k$. Its boundary value is
        zero in the limiting sense. For $l=0$, this follows from $\sigma_k=0$ on the
        finite cone boundary. When $l\ge1$, let $\lambda^{(j)}\in\Gamma_k$ tend to
        $\partial\Gamma_k$. If $\sigma_l(\lambda^{(j)})$ stays positive, then
        $\sigma_k/\sigma_l\to0$. If $\sigma_l(\lambda^{(j)})\to0$, the Maclaurin
        inequality
        \[
                \sigma_k(\lambda^{(j)})
                \le C\sigma_l(\lambda^{(j)})^{k/l}
        \]
        gives the same conclusion. Since $h$ is positive on the compact annulus, the
        required nondegeneracy condition follows. Degree-one homogeneity supplies
        the required large-ray condition.

        The Euclidean bounded-domain subsolution theorem in
        \cite{Guan2023Dirichlet} therefore applies on $D$. The full-boundary
        subsolution supplies the boundary estimates on both $\Gamma_-$ and
        $\Gamma_+$; the continuity method gives existence, and comparison gives
        uniqueness. The finite-regularity statement follows by standard smooth approximation, using
        the global $C^2$ estimate, Evans--Krylov, and boundary Schauder estimates uniformly.
\end{proof}

For Proposition~\ref{prop:truncated-annulus-solvability}, take
\[
        \Omega_-=\Omega,
        \qquad
        \Omega_+=\{\rho<R\},
        \qquad
        D=D_R,
\]
\[
        h=f_R,
        \qquad
        \varphi_-=\phi,
        \qquad
        \varphi_+=\overline u_R^\infty\big|_{\{\rho=R\}}.
\]
The function $\underline u_R$ constructed in Section~\ref{sec:construction}
is the required full-boundary admissible subsolution. Hence the
preceding proposition gives the solution of \eqref{eq:annular-problem} with
the prescribed inner and outer Dirichlet values.

\end{document}